\documentclass[bj]{imsart}

\RequirePackage[OT1]{fontenc}
\usepackage{amsthm,amsmath,amsfonts,amssymb}
\usepackage[numbers]{natbib}
\RequirePackage[colorlinks,citecolor=blue,urlcolor=blue]{hyperref}

\startlocaldefs
\usepackage[utf8]{inputenc}
\usepackage{microtype}
\usepackage{dsfont}
\usepackage{color}
\usepackage{enumerate}

\newtheorem{theorem}{Theorem}[section]
\newtheorem{proposition}[theorem]{Proposition}
\newtheorem{lemma}[theorem]{Lemma}
\newtheorem{corollary}[theorem]{Corollary}

\theoremstyle{remark}
\newtheorem{remark}[theorem]{Remark}

\definecolor{auburn}{rgb}{0.43, 0.21, 0.1}
\definecolor{britishracinggreen}{rgb}{0.0, 0.26, 0.15}
\definecolor{burntumber}{rgb}{0.54, 0.2, 0.14}
\definecolor{carmine}{rgb}{0.59, 0.0, 0.09}
\definecolor{aurometalsaurus}{rgb}{0.43, 0.5, 0.5}
\definecolor{gray}{rgb}{0.4, 0.4, 0.4 }
\newcommand{\js}[1]{\textcolor{burntumber}{\small\sffamily [JS: {#1}]}}

\newcommand{\mdf}[1]{\textcolor{red}{#1}}

\renewcommand{\le}{\leqslant}
\renewcommand{\leq}{\leqslant}
\renewcommand{\ge}{\geqslant}
\renewcommand{\geq}{\geqslant}

\renewcommand{\emptyset}{\varnothing}

\newcommand{\graph}{\operatorname{gph}}
\newcommand{\gph}{\graph}
\newcommand{\glimsup}{\operatornamewithlimits{g-lim\,sup}}
\newcommand{\gliminf}{\operatornamewithlimits{g-lim\,inf}}
\newcommand{\toto}{\rightrightarrows}

\newcommand{\supp}{\operatorname{spt}}
\newcommand{\spt}{\supp}
\newcommand{\epi}{\operatorname{epi}}
\newcommand{\cl}{\operatorname{cl}}

\newcommand{\intr}{\operatorname{int}}
\newcommand{\bnd}{\operatorname{bnd}}
\newcommand{\dom}{\operatorname{dom}}
\newcommand{\proj}{\operatorname{proj}}

\newcommand{\conv}{\operatorname{conv}}
\newcommand{\ball}{\mathds{B}}
\newcommand{\oball}{\ball^\circ}

\newcommand{\prob}{\mathbb{P}}
\newcommand{\Sphere}{\mathbb{S}_{d-1}}
\newcommand{\Fell}{\mathcal{F}}

\newcommand{\reals}{\mathbb{R}}
\newcommand{\NN}{\mathbb{N}}
\newcommand{\id}{\mathrm{Id}}
\newcommand{\1}{\mathds{1}}

\newcommand{\Mz}{\mathcal{M}_{0}}
\newcommand{\Rd}{\reals^{d}}

\newcommand{\eps}{\varepsilon}
\newcommand{\diff}{\mathrm{d}}
\newcommand{\point}{\,\cdot\,}
\newcommand{\diag}{\operatorname{diag}}

\newcommand{\gto}{\raisebox{-0.5pt}{\,\scriptsize$\stackrel{\raisebox{-0.5pt}{\mbox{\tiny $\mathrm{g}$}}}{\longrightarrow}$}\,}
\newcommand{\vto}{\raisebox{-0.5pt}{\,\scriptsize$\stackrel{\raisebox{-0.5pt}{\mbox{\tiny $0$}}}{\longrightarrow}$}\,}
\newcommand{\wto}{\raisebox{-0.5pt}{\,\scriptsize$\stackrel{\raisebox{-0.5pt}{\mbox{\tiny $\mathrm{w}$}}}{\longrightarrow}$}\,}

\newcommand{\inpr}[1]{\langle{#1}\rangle}
\newcommand{\norm}[1]{\left\lvert{#1}\right\rvert}
\newcommand{\abs}[1]{\left\lvert{#1}\right\rvert}

\newcommand{\bpsi}{\bar{\psi}}
\newcommand{\sphere}{\mathbb{S}}

\newcommand{\Prob}{\mathcal{P}}

\newcommand{\expec}{\mathbb{E}}

\endlocaldefs

\begin{document}

\begin{frontmatter}
\title{Tails of optimal transport plans for regularly varying probability measures}
\runtitle{Optimal transport plans for regularly varying probability measures}

\begin{aug}
\author{\fnms{Cees} 
	\snm{de Valk}%
	\thanksref{a}}
\and
\author{\fnms{Johan} 
	\snm{Segers}%
	\thanksref{b}}

\address[a]{Koninklijk Nederlands Meteorologisch Instituut, 
	PO Box 201, 
	NL-3730 AE De Bilt, 
	Netherlands. E-mail: \href{mailto:cees.de.valk@knmi.nl}{\ttfamily\upshape cees.de.valk@knmi.nl}}

\address[b]{UCLouvain, LIDAM/ISBA, 
	Voie du Roman Pays 20, 
	1348 Louvain-la-Neuve, 
	Belgium. E-mail: \href{mailto:johan.segers@uclouvain.be}{\ttfamily\upshape johan.segers@uclouvain.be}}
	
\runauthor{C. de Valk and J. Segers}
	
\affiliation{Koninklijk Nederlands Meteorologisch Instituut and UCLouvain}
\end{aug}

\begin{abstract}
For the basic case of $L_2$ optimal transport between two probability measures on a Euclidean space, the regularity of the coupling measure and the transport map in the tail regions of these measures is studied. For this purpose, Robert McCann's classical existence and uniqueness results are extended to a class of possibly infinite measures, finite outside neighbourhoods of the origin. For convergent sequences of pairs of such measures, the stability of the multivalued transport maps is considered, and a useful notion of locally uniform convergence of these maps is verified under light assumptions. Applied to regularly varying probability measures, these general results imply the existence of tail limits of the transport plan and the coupling measure, these objects exhibiting distinct types of homogeneity.
\end{abstract}


\end{frontmatter}

\section{Introduction}

The tail regions of a probability distribution are particularly important in many applications: they may be associated with serious consequences like damage, casualties, and financial loss, or in other cases, with exceptional benefits. 
Approximations or estimates of the tail of a probability measure can be improved if a suitable notion of regularity applies to the tail. The formulation and application of tail regularity conditions is the subject of extreme value theory. 

In this paper, we address the optimal transport between probability measures $\mu$ and $\nu$ on $\Rd$, and we seek to determine to what extent an optimal transport map and an associated optimal coupling measure between $\mu$ and $\nu$ inherit the tail regularity of these measures. The motivation to study this problem stems partly from potential applications of optimal transport such as calibration (with an optimal transport plan as generalisation of the matching of quantiles as in e.g. \cite{Squintu}) 
and optimal transport-based multivariate quantiles and related concepts recently introduced in \citep{Chernozhukov, dBCAHM18}.

We focus on a common tail regularity assumption known as multivariate regular variation \citep{Laurens__boek, Resnick_book, Resnick-heavy}: if $Y$ is a random vector with probability measure $\nu$, then the tail of its radius $\norm{Y}$ can be approximated by a power law, and given that $\norm{Y}$ is large, the direction $Y / \norm{Y}$ is almost independent of the radius $\norm{Y}$.

As regular variation involves the convergence of certain finite intensity measures built from  $\mu$ and $\nu$ to intensity measures  $\bar{\mu}$ and $\bar{\nu}$, one might suspect that a basic stability result along the lines of Theorem~5.20 in \citet{Villani} will also apply to optimal transport maps between the intensity measures.

However, we run into a major obstacle: the limiting measures $\bar{\mu}$ and $\bar{\nu}$ are not even finite. Convergence takes place in the space $\Mz(\Rd)$ of Borel measures which are finite on complements of neighbourhoods of the origin \citep{HultLindskog2006}. As this situation is not covered in the literature, we first address basic questions of existence, uniqueness, representation and stability of optimal couplings between possibly infinite measures in the space $\Mz(\Rd)$.

In order to limit the complexity and length of this paper, we consider the case of a cost function which is quadratic in the Euclidean norm of the displacement. For the case of probability measures, \citet{McCann} in particular treated optimal transport with quadratic cost in full generality. Moment assumptions are avoided, and the transportation cost may be infinite. This is the generalized optimal transport problem addressed for a wider class of cost functions in \citep[Theorem~10.42]{Villani}. In Sections~\ref{sec:unique} and~\ref{sec:stability}, we extend this treatment to the possibly infinite measures in the space $\Mz(\Rd)$. 
In particular for the uniqueness issue, this extension is not trivial.

The space $\Mz(\Rd)$ comes equipped with a topology accommodating convergence of possibly infinite measures. For sequences of pairs of measures converging in $\Mz(\Rd)$, we show in Section~\ref{sec:stability} the stability of the sequence of coupling measures with cyclically monotone supports and of the maps that support those coupling measures. Furthermore, we demonstrate a form of locally uniform convergence of (multi-valued) transportation maps under mild assumptions.
Weak convergence of probability measures being equivalent to their convergence in $\Mz(\Rd)$, our stability result complements those in, for example, \citet[Section~3]{CAMT97}, \citet[Chapter~5]{Villani}, and \citet[Section~2]{dBL:2019}.

In Section~\ref{sec:RV}, we apply these basic results to the case of regularly varying probability measures. We show that the sequences of suitably rescaled optimal coupling measures and associated rescaled transportation maps admit converging subsequences, and, under mild additional assumptions, that the sequences themselves converge. Moreover, the limit objects satisfy distinct types of homogeneity.

\section{Preliminaries}
\label{sec:prelimin}

We will be needing concepts and results from variational analysis and measure transportation. This section serves to recall some basic facts and to fix some notation. The interior, closure and boundary of a subset $A$ of a topological space are denoted by $\intr (A)$, $\cl(A)$ and $\bnd(A)$, respectively. Let $\conv(A)$ denote the convex hull of a subset $A$ of a real vector space. The indicator function of $A$ is $\1_A(x) = 1$ if $x \in A$ and $\1_A(x) = 0$ if $x \not\in A$; sometimes we also write $\1(x \in A)$. The identity function on a space clear from the context is denoted by $\id$. The scalar product and the Euclidean norm on Euclidean space are denoted by $\inpr{\point, \point}$ and $\norm{\point}$, respectively. 
The unit sphere in $\Rd$ is denoted by $\Sphere = \{x \in \Rd : \norm{x} = 1 \}$, while the open and closed unit balls in Euclidean space are $\oball = \{v : \norm{v} < 1 \}$ and $\ball = \{v : \norm{v} \le 1\}$, respectively, where the dimension will be clear from the context. For a point $x$ and a scalar $r > 0$, the sets $\oball(x, r) = x + r \oball = \{ x + rv : \norm{v} < 1 \}$ and $\ball(x, r) = x + r \ball = \{ x + rv : \norm{v} \le 1 \}$ are the open and closed balls with centre $x$ and radius $r$.
More generally, for subsets $A$ and $B$ of some Euclidean space, we put $A + B = \{a + b : a \in A, \, b \in B \}$. 
The Hausdorff distance between two non-empty bounded subsets $K$ and $L$ of Euclidean space is
\begin{equation}
\label{eq:hausdorff}
	d_H(K, L) = \inf \{ \eps \ge 0 : 
		K \subset L + \eps \ball \text{ and } L \subset K + \eps \ball 
	\}.
\end{equation}
It is a metric when restricted to the collection of non-empty compact subsets.
The set of all Borel probability measures on $\reals^k$ is denoted by $\Prob(\reals^k)$.

\subsection{Convex functions and their subdifferentials}
\label{sec:prelimin:convex}

We consider convex functions $\psi : \Rd \to \reals \cup \{+\infty\}$ whose domain, $\dom \psi = \{ x \in \Rd : \psi(x) < +\infty \}$, is not empty; such convex functions are called proper \citep[Section~4]{Rockafellar}, a property which we will assume throughout. A convex function $\psi$ is said to be closed if its epigraph, $\epi \psi = \{ (x, \lambda) \in \Rd \times \reals : \psi(x) \le \lambda \}$, is closed, or equivalently, if $\psi$ is lower semicontinuous \citep[Section~7]{Rockafellar}. A convex function can be closed by taking its lower semicontinuous minorant. This operation amounts to closing the function's epigraph in $\Rd \times \reals$ and only affects its values on the boundary of its domain.

The subdifferential of a convex function $\psi$ at a point $x \in \Rd$ is the set $\partial\psi(x)$ of all points $y \in \Rd$ such that
\[
	\forall z \in \Rd, \qquad \psi(z) \ge \psi(x) + \inpr{y, z-x}.
\]
The domain of $\partial\psi$, notation $\dom \partial\psi$, is the set of all $x$ at which $\partial\psi(x)$ is not empty.
The function $\psi$ is differentiable at $x$ if and only $\partial\psi(x)$ is a singleton, and then $\partial\psi(x) = \{ \nabla\psi(x) \}$. This happens at all $x$ in the interior of $\dom(\psi)$ minus a set of Hausdorff dimension at most $d-1$ \citep{AndersonKlee} and thus of zero Lebesgue measure \citep[Theorem~25.5]{Rockafellar}. The subdifferential $\partial\psi(x)$ is empty if $\psi(x) = +\infty$ and nonempty if $x \in \intr(\dom \psi)$. 
Combining \citet[Theorem~6.3, Theorem~23.4, and Corollary~25.1.1]{Rockafellar}, we have in fact
\begin{equation}
\label{eq:domains}
	\left.
	\begin{array}{rcrcr}
		\dom \nabla \psi &\subset& \intr(\dom \partial \psi) &=& \intr(\dom \psi), \\
		&& \dom \partial \psi\phantom{)} &\subset& \dom \psi\phantom{)}, \\
		&& \cl(\dom \partial \psi) &=& \cl(\dom \psi).
	\end{array}	
	\right\}
\end{equation}

%

We view $\partial\psi$ as a multivalued map (or map in short), that is, a map from $\Rd$ into the power set of $\Rd$; notation $\partial\psi : \Rd \rightrightarrows \Rd$. The graph of a map $S : \mathcal{X} \rightrightarrows \mathcal{Y}$ is
\[ 
	\gph(S) = \{(x, y) \in \mathcal{X} \times \mathcal{Y} : y \in S(x) \}. 
\]
Note that the graph of $S$ is a subset of $\mathcal{X} \times \mathcal{Y}$ rather than of the Cartesian product of $\mathcal{X}$ and the power set of $\mathcal{Y}$. The domain of $S$ is $\dom S = \{ x \in \mathcal{X} : S(x) \neq \varnothing \}$. For two maps $S, T : \mathcal{X} \rightrightarrows \mathcal{Y}$, we say that $S \subset T$ if $S(x) \subset T(x)$ for all $x \in \mathcal{X}$, or equivalently, if $\gph(S) \subset \gph(T)$.

The derivative of a convex function on a real interval is non-decreasing, and non-decreasing functions constitute the cheapest way for moving masses on the real line with respect to squared distance as cost function. On Euclidean space, subdifferentials of convex functions play a similar role.
%
A multivalued map $S : \Rd \rightrightarrows \Rd$ is called cyclically monotone if, for every finite collection of pairs $(x_1, y_1)$, $\ldots$, $(x_n, y_n)$ such that $y_i \in S(x_i)$ for all $i = 1, \ldots, n$, we have, writing $y_{n+1} = y_1$,
\begin{equation}
\label{eq:cm:d}
	\sum_{i=1}^n \norm{x_i - y_i}^2 \leq \sum_{i=1}^n \norm{x_i - y_{i+1}}^2.
\end{equation}
Eq.~\eqref{eq:cm:d} says that moving a unit point mass from location $x_i$ to location $y_i$ for all $i = 1, \ldots, n$ minimizes the total quadratic cost among all plans moving unit point masses from locations $x_1, \ldots, x_n$ to locations $y_1, \ldots, y_n$ (since any permutation of $\{1, \ldots, n\}$ can be decomposed into cycles on disjoint subsets). 

A map $S : \Rd \rightrightarrows \Rd$ is called monotone if \eqref{eq:cm:d} holds for $n= 2$. 
A map $S : \Rd \rightrightarrows \Rd$ is called maximal (cyclically) monotone if it is (cyclically) monotone and if it is not strictly contained in another (cyclically) monotone map. The graphs of maximal (cyclically) monotone maps are necessarily closed. The relevance of subdifferentials of convex functions for optimal measure transport comes from the following characterization \citetext{\citealp[Theorem~24.8 and~24.9]{Rockafellar}; \citealp[Theorem~12.25]{Rockafellar-Wets}}.

\begin{theorem}[Rockafellar's Theorem]
	\label{thm:Rock}
	A map $S : \Rd \rightrightarrows \Rd$ is cyclically monotone if and only if it is contained in the subdifferential of a closed convex function. It is maximal cyclically monotone if it is equal to the subdifferential of a closed convex function.
\end{theorem}

In addition, subdifferentials of closed convex functions are maximal monotone, see \citep[Corollary~31.5.2]{Rockafellar} or \citep[Theorem~12.17]{Rockafellar-Wets}. 
This property does not follow immediately from the fact that they are maximal cyclically monotone and is greatly helpful when studying graphical convergence of sequences of subdifferentials (Appendix~\ref{sec:gconv}).

A subset $T$ of $\Rd \times \Rd$ can be identified with the multivalued map sending $x \in \Rd$ to $\{ y \in \Rd : (x, y) \in T \}$; clearly, $T$ is the graph of this map. Such a set $T$ will be said to enjoy one of the monotonicity properties above if the associated map possesses the property.

To study the asymptotic properties of a sequence of subdifferentials of closed convex functions, we will rely on the concept of graphical convergence. A sequence of maps $S_n : \reals^k \toto \reals^\ell$ converges graphically to a map $S : \reals^k \toto \reals^\ell$ if the graphs of $S_n$ converge to the graph of $S$ as subsets of $\reals^k \times \reals^\ell$ in the sense of Painlev\'e--Kuratowski. The graphical limit of a sequence of subdifferentials of closed convex functions is again the subdifferential of a closed convex function, and the convergence takes place locally uniformly at points in the domain of the gradient of the limit function. Graphical convergence and relevant properties thereof are reviewed in Appendix~\ref{sec:gconv}.

Cyclically monotone maps provide solutions to optimal assignment problems via Eq.~\eqref{eq:cm:d}. They also play a key role in the theory of optimal transport, as explained next.

\subsection{Measure transportation with quadratic distance as cost}
\label{sec:OT}

For two Borel measures $\mu$ and $\nu$ on $\reals^{d}$, not necessarily finite, a coupling measure is a Borel measure $\pi$ on $\reals^{d}\times\reals^{d}$ satisfying $\pi(A\times\reals^{d})=\mu(A)$ and $\pi(\reals^{d}\times A)=\nu(A)$ for every Borel set $A\subset\reals^{d}$. The collection of coupling measures of $\mu$ and $\nu$ is denoted by $\varPi(\mu,\nu)$. We call $\mu$ and $\nu$ the (left and right) marginals of $\pi$.

Suppose $\mu$ and $\nu$ are finite Borel measures on $\reals^d$ with equal, non-zero mass. An optimal transport plan with respect to the squared distance as cost function is a coupling measure $\pi\in\varPi(\mu,\nu)$ solving the Kantorovich problem
\begin{equation}
\label{eq:OptTrans}
	\int_{\Rd \times \Rd} \norm{x - y}^2 \, \diff\pi(x,y)
	=
	\inf_{\pi'\in\varPi(\mu,\nu)} \int_{\Rd \times \Rd} \norm{x - y}^2 \, \diff\pi'(x,y).
\end{equation}
An optimal transport plan exists if $\mu$ and $\nu$ have finite second-order moments, i.e., $\int_{\Rd} \norm{x}^{2} \diff\mu(x) < \infty$ and $\int_{\Rd} \norm{y}^{2} \diff\nu(y) < \infty$ \citep[Theorem~2.12]{Villani-Topics}. The infimum in~\eqref{eq:OptTrans} is the square of the quadratic Wasserstein distance between $\mu$ and $\nu$. Theory for more general cost functions and probability measures defined on broader classes of spaces is extensively covered in \citet{Villani}.

The condition that $\mu$ and $\nu$ have finite second-order moments can be relaxed if we allow the optimal transport cost in \eqref{eq:OptTrans} to be infinite. In this case, all transport plans have ``optimal'' cost, but a meaningful generalisation of an optimal transport plan can still be defined as a coupling measure $\pi$ with cyclically monotone support; see, e.g., \citep[Section~2.3]{Villani-Topics}. Recall that the support $\spt(\mu)$ of a Borel measure $\mu$ on a metric space is the collection of points in the space with the property that every neighbourhood of the point receives positive mass. The support is necessarily a closed set, and it is 
the smallest closest subset of the space of which the complement is a null set by the given measure. 
\citet{McCann}, Theorem~6] proved that a coupling measure $\pi \in \varPi(\mu,\nu)$ with cyclically monotone support always exists, and by Rockafellar's Theorem~\ref{thm:Rock}, its support, $\supp \pi$, must be a subset of the graph of the subdifferential $\partial \psi$ of a closed convex function $\psi$.

A variant of the Kantorovich problem is the Monge problem, in which the coupling measure $\pi$ is required to be generated by a Borel-measurable map $T$: 
\begin{equation}
\label{eq:Monge}
\pi(B)= \mu(\{x: (x,T(x)) \in B\})
\end{equation}
for Borel $B \subset \Rd \times \Rd$ and therefore, $\nu(A)= \mu(\{x: T(x) \in A\}) = \mu(T^{-1}(A))$ for Borel $A \subset \Rd$. The latter is usually written as $\nu= T_\# \mu$: $\nu$ is the push-forward of $\mu$ by $T$. Similarly, we write \eqref{eq:Monge} as $\pi= (\id \times T)_\# \mu$.
Its support, $\supp \pi$, is a subset of the closure in $\Rd \times \Rd$ of the graph of $T$. 



Proposition~10 in \citep{McCann} and Theorem~\ref{thm:Rock} above 
imply that if $\mu$ vanishes on (Borel) sets of Hausdorff dimension $d-1$, then any coupling measure $\pi \in \Pi(\mu,\nu)$ with cyclically monotone support satisfies  $\pi = (\id \times \nabla \psi)_\# \mu$ and thus $\nu = \nabla\psi_\# \mu$ for some closed convex function $\psi$ satisfying $\supp \pi \subset \gph \partial \psi$ (and in fact, for all such $\psi$).
Furthermore, the same condition on $\mu$ implies that the map $\nabla\psi$ is  uniquely determined up to $\mu$-null sets:
\begin{theorem}[\citet{McCann}, Main Theorem]
	\label{thm:McCann}
	Let $\mu, \nu \in \Prob(\Rd)$ and suppose $\mu$ vanishes on (Borel) subsets of $\Rd$ having Hausdorff dimension $d-1$. Then there exists a convex function $\psi$ on $\Rd$ whose gradient $\nabla \psi$ pushes $\mu$ forward to $\nu$. Although $\psi$ is not unique, the map $\nabla \psi$ is uniquely determined $\mu$-almost everywhere.
\end{theorem}
The uniqueness of the gradient $\nabla \psi$ in Theorem~\ref{thm:McCann} together with the preceding description of coupling measures with cyclically monotone support
implies  uniqueness of the coupling measure $\pi$ with cyclically monotone support
\citep[Corollary~14]{McCann}.

These results are formulated for pairs of probability measures, but obviously, they remain true for pairs of Borel measures with equal, finite, non-zero mass.
Similar results were derived earlier under more restrictive assumptions in \citet{KnottSmith} and \citet{Brenier, brenier:1991}.

\section{Cyclically monotone transports between infinite measures}
\label{sec:unique}

Motivated by the study of tail limits of cyclically monotone transport plans between regularly varying probability measures in Section~\ref{sec:RV}, we seek to extend McCann's theory as sketched in Section~\ref{sec:OT} to a certain space of Borel measures on $\Rd$ with possibly infinite mass. Let $\Mz(\Rd)$ be the set of all Borel measures $\mu$ on $\Rd \setminus \{0\}$ that are finite on complements of neighbourhoods of the origin, that is, such that $\mu(\Rd \setminus r\ball)$ is finite for all $r > 0$.
The existence of the gradient of a convex function pushing one such measure to another one will follow from an approximation argument involving a sequence of finite measures, see Theorem~\ref{thm:existence:Mz} below. We first establish the uniqueness of such a gradient. Note that McCann's arguments do not and cannot readily extend to arbitrary infinite measures, not even to $\sigma$-finite ones: the Lebesgue measure on Euclidean space, for instance, is invariant under translations, violating uniqueness of the gradient.

\begin{theorem}[Uniqueness]
	\label{thm:unique}
	Let $\mu \in \Mz(\Rd)$ be a non-zero measure and suppose that $\mu$ vanishes on all sets of Hausdorff dimension at most $d-1$. 
	Let $\psi$ and $\phi$ be convex functions $\Rd \to \reals \cup \{+\infty\}$ that are finite $\mu$-almost everywhere. If $\nabla\psi_\# \mu = \nabla\phi_\# \mu \in \Mz(\Rd)$, then $\nabla\psi = \nabla \phi$ $\mu$-almost everywhere. In fact, $\nabla\psi(x) = \nabla\phi(x)$ for every $x \in \spt \mu$ at which the two gradients are defined.
\end{theorem}

The proof of Theorem~\ref{thm:unique} requires several steps and makes extensive use of the property that measures in $\Mz(\Rd)$ are finite on sets bounded away from the origin. 
In contrast, the following result only requires $\sigma$-finiteness, and its proof remains close to the one of \citep[Proposition~10]{McCann}.

\begin{proposition}[Representation]
	\label{prop:rprgmm}
	Let the support of the (possibly infinite) Borel measure $\pi$ on $\Rd \times \Rd$ be contained in the graph of the subdifferential $\partial\psi$ of the closed convex function $\psi$. Let $\mu$ and $\nu$ denote the marginals of $\pi$, so that $\pi \in \varPi(\mu, \nu)$.
	If $\mu$ is $\sigma$-finite and vanishes on (Borel) sets of Hausdorff dimension $d-1$, then $\nabla\psi$ is defined $\mu$-almost everywhere and $\pi = (\id \times \nabla\psi)_\# \mu$, implying $\nabla\psi_\# \mu = \nu$.
\end{proposition}

\begin{remark}
	\label{rem:unique}
	In Theorem~\ref{thm:unique}, the condition that the measure $\nabla\psi_\# \mu = \nabla\phi_\# \mu$ belongs to $\Mz(\Rd)$ is used only to show via Lemma~\ref{lem:Psi-finite} that $\psi(0)$ and $\phi(0)$ are finite in case $\mu$ has infinite mass. A weaker condition was thus possible, but since the focus in our paper is on the space $\Mz(\Rd)$, we have opted for the current formulation.
\end{remark}

\subsection*{Proofs}

\begin{proof}[Proof of Theorem~\ref{thm:unique}]
If $0\notin{\supp\mu}$, then ${\supp\mu}$, being closed, is bounded away from $0$. But as $\mu\in\Mz(\Rd)$, we must have $0 < \mu(\Rd) < +\infty$ and we can resort to the uniqueness proof in \citet[pp.~318--319]{McCann}. Therefore, consider the case that $\mu(\Rd) = +\infty$ and thus $0 \in {\supp\mu}$. 


\medskip{}

\noindent\emph{Step 1.} ---
Since $\psi$ and $\phi$ are finite $\mu$-a.e., ${\supp\mu}$ must lie in $\cl(\dom\psi)\cap\cl(\dom\phi)$. 
Both closures are convex.
Therefore, their boundaries have Hausdorff dimension at most $d-1$, so they are $\mu$-null sets. Since a convex function is differentiable in the interior of its domain except perhaps at a set of Hausdorff dimension at most $d-1$, we find that $\psi$ and $\phi$ are differentiable almost everywhere. Moreover, we may take closed convex functions for $\psi$ and $\phi$, since closure only affects their values on the boundaries of their domains. 
By Lemma~\ref{lem:Psi-finite}, $\psi$ and $\phi$ attain their minima at $0$, so that $\psi(0)$ and $\phi(0)$ are certainly finite. Therefore, without losing generality, take $\psi(0)= \phi(0)= 0$.
Furthermore, we can restrict attention to
\begin{equation}
\label{eq:calD}
	\mathcal{D} := \intr(\dom\psi) \cap \intr(\dom\phi) = \intr(\dom\psi\cap\dom\phi)
\end{equation}
and its subsets [see~\eqref{eq:domains}]
\begin{align}
\label{eq:domnablas}
V & :=\dom(\nabla\psi) \cap \dom(\nabla\phi), \\
\nonumber
W & :=\left\{ y\in V:\,\nabla\psi(y)\neq\nabla\phi(y)\right\} ,\\
\label{eq:psiphidiff0}
A & :=\left\{ y\in\mathcal{D}:\,\psi(y)-\phi(y)\neq 0 \right\} .
\end{align}

By \citet[Theorem~2.2]{A-A} (see also \citep{AndersonKlee}),
$\mathcal{D}\setminus V$ is contained in the union of two sets with Hausdorff dimension at most $d-1$, so 
\begin{equation}
	\label{eq:muDV}
	\mu(\mathcal{D}\setminus V) = 0.
\end{equation}

Our claim is equivalent to $\mu(W)=0$. 
Since the gradient of a convex function is continuous on its domain, the set $W$ is relatively open in $V$, that is, $W = V \cap U$ for some open $U \subset \Rd$. Since $V \subset \mathcal{D}$ and since $\mathcal{D}$ is open, we may assume that $U \subset \mathcal{D}$; otherwise, intersect $U$ with $\mathcal{D}$. As the complement of $V$ in $\mathcal{D}$ is a $\mu$-null set, we have $\mu(W) = \mu(U)$. From $\mu(W) = 0$, it then follows that $\mu(U) = 0$ and thus that $U$ is disjoint from $\spt \mu$. But then $W$ is disjoint from $\spt \mu$ as well. This implies that $\nabla \psi(y) = \nabla \phi(y)$ for every $y \in \spt \mu$ in the domain of $\nabla \psi$ and $\nabla \phi$.

To prove that $\mu(W)=0$, we assume the opposite:
\begin{equation}
\label{eq:wrongassumption}
	\mu(W) > 0, 
\end{equation}
and we will construct a subset of $\Rd$ that receives different masses from the two push-forwards $\nabla\psi_{\#}\mu$ and $\nabla\phi_{\#}\mu$, in contradiction to the assumption of the theorem.

\medskip{}

\noindent\emph{Step 2.} ---
By Lemma~\ref{lem:smallset}, the set
\[
	W \setminus A = \{ y\in V : 
		\psi(y) - \phi(y) = 0 \text{ and } \nabla\psi(y)\ne\nabla\phi(y)
	\}
\]
has Hausdorff dimension at most $d-1$.
Therefore, $\mu(W\setminus A)=0$. By \eqref{eq:wrongassumption}, we obtain 
\begin{equation}
\label{eq:wrongassumption2}
	\mu(W \cap A) > 0. 
\end{equation}

\medskip{}

\noindent\emph{Step 3.} --- 
Recall $\mathcal{D}$ in \eqref{eq:calD}. For $t\in[-\infty, +\infty]$, define
\begin{align*}
A_{t}^{-} &:= \left\{ y\in\mathcal{D}:\,\psi(y)-\phi(y)<t\right\} ,\\
A_{t}^{+} &:= \left\{ y\in\mathcal{D}:\,\psi(y)-\phi(y)>t\right\} ,\\
A_{t}^{0} &:= \left\{ y\in\mathcal{D}:\,\psi(y)-\phi(y)=t\right\} .
\end{align*}
Then $A_{0}^{-}\cup A_0^{+}$ is a partition of $A$ in \eqref{eq:psiphidiff0} and thus, by \eqref{eq:wrongassumption2},
\[ 
	\mu(W\cap A_{0}^{-})+\mu(W\cap A_{0}^{+})=\mu(W\cap A)>0. 
\]
Therefore, $\mu(W\cap A_{0}^{-})>0$ or $\mu(W\cap A_{0}^{+})>0$. Without loss of generality, assume
\begin{equation}
\label{eq:wrongassumption3}
	\mu(W\cap A_{0}^{-}) > 0.
\end{equation}
In the other case, just switch the roles of $\psi$ and $\phi$ in the next steps.
\medskip{}

\noindent\emph{Step 4.} --- 
For every $t<0$, the closure of the set $A_{t}^{-}$ does not contain the origin, since $\psi$ and $\phi$ are continuous and $\psi(0) - \phi(0) = 0 > t$. As $\mu\in\mathcal{M}_{0}(\Rd)$, we must have 
\[ 
	\forall t < 0, \qquad \mu(A_{t}^{-}) < +\infty. 
\]

\medskip{}

\noindent\emph{Step 5.} --- 
Recall $V$ in \eqref{eq:domnablas}. We claim that there exists $z \in V$ with the following three properties:
\begin{itemize}
	\item $t := \psi(z) - \phi(z) < 0$;
	\item $\nabla{\psi}(z) \ne \nabla{\phi}(z)$;
	\item for all $\delta > 0$,
	\begin{equation} 
	\label{eq:mu_pos}
	\mu \left( A_{t}^{-} \cap  \oball(z, \delta) \right) > 0,
	\end{equation}
	where $\oball(\point,\point)$ denotes the open ball with given centre and radius.
\end{itemize}

To prove the claim, define, for every $t < 0$, the sets
\begin{align*}
	\Omega_{t}
	&:=
	A_{t}^{0}\cap W\cap\supp\mu \\
	&\phantom{:}=
	\{ x \in V : \psi(x) - \phi(x) = t \text{ and } \nabla\psi(x) \ne \nabla\phi(x) \} \cap \supp \mu, \\
	\Xi_{t}
	&:=
	\{ x \in \Omega_{t} : \exists \delta > 0, \mu ( A_{t}^{-} \cap \oball(x, \delta) ) = 0 \}.
\end{align*}
We need to show that there exists $t < 0$ such that $\Omega_{t} \setminus \Xi_{t} \neq \emptyset$. To this end, define
\begin{align*}
	\mathcal{T}_{\Omega} &:= \{ t \in (-\infty, 0) : \Omega_{t} \neq \emptyset \}, \\
	\mathcal{T}_{\Xi} &:= \{ t \in (-\infty, 0) : \Xi_{t} \neq \emptyset \}.
\end{align*}
We will show that $\mathcal{T}_{\Omega}$ is uncountable while $\mathcal{T}_{\Xi}$ is at most countable, so that there must exist (uncountable many) $t < 0$ such that $\Omega_{t}$ is nonempty and $\Xi_{t}$ is empty.

By Lemma~\ref{lem:smallset}, for every $t < 0$, the set $\Omega_{t}$ has Hausdorff measure at most $d-1$ and is therefore a $\mu$-null set. Still, their union over all $t$ less than $0$ is
\[
	\textstyle\bigcup_{t < 0} \Omega_t
	=
	\textstyle\bigcup_{t < 0} A_{t}^{0} \cap W \cap \supp \mu
	=
	A_{0}^{-} \cap W \cap \supp \mu,
\]
a set receiving positive mass from $\mu$ by \eqref{eq:wrongassumption3}. Therefore, $\mathcal{T}_\Omega$ must be uncountable.

We have $\Xi_{t} = \bigcup_{n \in \NN} \Xi_{t,n}$ and thus $\mathcal{T}_{\Xi} = \bigcup_{n \in \NN} \mathcal{T}_{n}$, where, for $n \in \NN$,
\begin{align*}
	\Xi_{t,n}
	&:=
	\{ x \in \Omega_{t} : \mu ( A_{t}^{-} \cap \oball(x, 2/n) ) = 0 \}, \\
	\mathcal{T}_{n} 
	&:= \{ t \in (-\infty, 0)  : \Xi_{t,n} \neq \emptyset \}.
\end{align*}
To show that $\mathcal{T}_{\Xi}$ is countable, it is sufficient to show that each $\mathcal{T}_{n}$ is countable.

Fix $n \in \NN$. Let $s, t \in \mathcal{T}_{n}$ with $s < t$ and let $x_{s} \in \Xi_{s,n}$ and $x_t \in \Xi_{t,n}$. As $\Xi_{s,n} \subset \Omega_s \subset \supp \mu$, the point $x_s$ belongs to the support of $\mu$ and can therefore not belong to the open $\mu$-null set $A_{t}^{-} \cap \oball(x_{t}, 2/n)$. But as $\psi(x_{s}) - \phi(x_{s}) = s < t$, we do have $x_{s} \in A_{t}^{-}$. Therefore, $x_{s}$ cannot be an element of $\ball(x_{t}, 2/n)$, that is, $\norm{x_{s} - x_{t}} \ge 2/n$. By the triangle inequality, the balls $\oball(x_s, 1/n)$ and $\oball(x_t, 1/n)$ must be disjoint. Define
\[
\forall t \in \mathcal{T}_{n}, \qquad
G_t := {\textstyle\bigcup}_{x \in \Xi_{t,n}} \oball(x, 1/n).
\]
Each $G_t$ is non-empty and open, and by the argument just given, $G_s \cap G_t = \emptyset$ whenever $s, t \in \mathcal{T}_{n}$ with $s < t$. Since $\Rd$ is separable, $\mathcal{T}_n$ is at most countable, as required. The claim made at the beginning of this step is proved.

\medskip{}

\emph{Step 6.} ---
For $z$ from Step 5, define the closed convex functions
$\varphi:=\phi-\phi(z)$ and $\bar{\varphi}:=\psi-\psi(z)$.
Then $\varphi(z)=\bar{\varphi}(z)$ and $\nabla\varphi(z)\neq\nabla\bar{\varphi}(z)$.
Let $t=\psi(z)-\phi(z)<0$ and consider the open set
\[
	M:=A_{t}^{-}=\{y\in\mathcal{D}:\varphi(y)>\bar{\varphi}(y)\}.
\]
By Steps~4 and~5, we have $0 < \mu(M) < +\infty$.

At this point, we can apply the reasoning of \citet[pp.~318--319]{McCann}. It shows
that $\partial\varphi(M)$ is a Borel set and that by Aleksandrov's
lemma \citep[Lemma~13]{McCann}, the set $Z:=[\nabla\bar{\varphi}]^{-1}(\partial\varphi(M))$
is at a nonzero distance from $z$ and satisfies $Z\subset M$.
Therefore, \eqref{eq:mu_pos} implies that $\mu(M\setminus Z)>0$,
and since $\mu(M)$ is finite, that $\mu(Z)<\mu(M)$. Hence, 
\begin{align*}
	[\nabla\bar{\varphi}_\#\mu](\partial\varphi(M)) 
	&=\mu \left([\nabla\bar{\varphi}]^{-1}(\partial\varphi(M))\right)=\mu(Z)<\mu(M)\\
	&\leq \mu\left([\nabla\varphi]^{-1}(\partial\varphi(M))\right)
	=[\nabla\varphi_\#\mu](\partial\varphi(M)),
\end{align*}
the last inequality deriving from \eqref{eq:muDV} and 
\[
M\subset(\mathcal{D}\setminus V)\cup[\nabla\varphi]^{-1}(\nabla\varphi(M))\subset(\mathcal{D}\setminus V)\cup[\nabla\varphi]^{-1}(\partial\varphi(M)).
\]
As a consequence, $\nabla\bar{\varphi}_\#\mu\neq\nabla\varphi_\#\mu$, in contradiction to the assumption of the theorem. The inequality \eqref{eq:wrongassumption} must therefore be false, and the proof is complete.
\end{proof}

\begin{lemma}
	\label{lem:Psi-finite}
	Let $\mu, \nu \in \Mz(\Rd)$ with $\mu(\Rd) = \nu(\Rd) = +\infty$ and let $\pi \in \varPi(\mu, \nu)$ be supported by $\gph( \partial \psi )$ for some closed convex function $\psi$. Then $\psi(0) = \inf_{x \in \Rd} \psi(x)$. This holds in particular if $\mu \in \Mz(\Rd)$ has infinite mass, $\psi$ is a closed convex function that is $\mu$-a.e.\ differentiable, and $\nu = \nabla\psi_{\#} \mu \in \Mz(\Rd)$.
\end{lemma}

\begin{proof}
	The coupling measure $\pi$ must have infinite mass but still belong to $\Mz(\Rd \times \Rd)$, see Lemma~\ref{lem:rlcmpct}(a). Therefore, necessarily $\pi(\eps \ball \times \eps \ball) = +\infty$ for every $\eps > 0$, which means that $\spt \pi$ meets $\eps \ball \times \eps \ball$ for every $\eps > 0$. But the support of a Borel measure is always closed, and thus $(0, 0) \in \spt \pi \subset \gph(\partial \psi)$, the latter inclusion holding by assumption. But then also $0 \in \dom \partial \psi \subset \dom \psi$ in view of \eqref{eq:domains}. By definition of the subdifferential we find
	\[
		\forall x \in \Rd, \qquad
		\psi(x) \ge \psi(0) + \inpr{x - 0, 0} = \psi(0),
	\]
	as required.
	
	If $\psi$ is $\mu$-a.e.\ differentiable and if $\nu = \nabla\psi_{\#} \mu$, then $\pi = (\id \times \nabla\psi)_{\#} \mu$ belongs to $\varPi(\mu, \nu)$ and the support of $\pi$ is contained in $\cl(\{ (x, \nabla\psi(x)) : x \in \dom \nabla\psi \})$, which is in turn a subset of $\gph( \partial \psi )$. Note that $\gph(\partial \psi)$ is closed since $\psi$ was assumed to be closed and thus $\partial \psi$ is maximal cyclically monotone by Rockafellar's theorem (Section~\ref{sec:prelimin:convex}).
\end{proof}

\begin{lemma}
	\label{lem:smallset}
	Let $\psi$ and $\phi$ be convex functions on $\Rd$. For every $t \in \reals$, the set
	\begin{equation}
	\label{eq:smallset}
		\{ 
			x \in \dom(\nabla\psi) \cap \dom(\nabla\phi) : 
			\psi(x) - \phi(x) = t \text{ and } \nabla\psi(x) \ne \nabla\phi(x)
		\}
	\end{equation}
	has Hausdorff dimension at most $d-1$.
\end{lemma}

\begin{proof}
	Let $X$ be the set in \eqref{eq:smallset}. By McCann's Implicit Function Theorem \citep[Theorem~17 and Corollary~19]{McCann}, every point $x$ in $X$ has an open neighbourhood $N_{x}$ such that the $(d-1)$-dimensional Hausdorff measure of the set
	\[
		\{ y \in N_{x} : \psi(y)-\phi(y) = t \}
	\] 
	is finite. The collection of sets $\{N_{x} : x \in X\}$ forms an open cover of $X$. By the Lindel\"{o}f property of $\Rd$, there exists a countable subcover, i.e., a countable set $Q \subset X$ such that $X$ is a subset of $\bigcup_{x \in Q}N_{x}$ and then also of 
	\[ 
		\textstyle\bigcup_{x \in Q} \{y \in N_{x} : \psi(y)-\phi(y)=t\}. 
	\]
	The latter set is a countable union of sets of Hausdorff dimension at most $d-1$ and has therefore Hausdorff dimension at most $d-1$ as well. 
\end{proof}

\begin{proof}[Proof of Proposition~\ref{prop:rprgmm}]
	The proof is almost identical to the one of Proposition~10 in \citet{McCann}. The first paragraph of that proof yields that $\psi$ is differentiable on a Borel set whose complement is a $\mu$-null set and that $\nabla\psi$ is Borel measurable.
	
	In the second paragraph of the cited proof, it is verified that $\pi$ and $(\id \times \nabla\psi)_\# \mu$ coincide on rectangles $M \times N$ of Borel sets $M, N \subset \Rd$. This property implies that $\pi$ and $(\id \times \nabla\psi)_\# \mu$ are equal, even if they are not necessarily finite, thanks to the hypothesis that $\mu$ is still $\sigma$-finite. Indeed, let $(B_n)_n$ be a sequence of Borel sets of $\Rd$ whose union is $\Rd$ and such that $\mu$ is finite on each $B_n$. Then $\pi(B_n \times \Rd) = \mu(B_n) = ((\id \times \nabla\psi)_\# \mu)(B_n \times \Rd)$ are finite too. The restrictions of $\pi$ and $(\id \times \nabla\psi)_\# \mu$ to $B_n \times \Rd$ coincide on Borel measurable rectangles, and thus they must be equal. As the union of the sets $B_n \times \Rd$ over all $n$ is equal to $\Rd \times \Rd$, we obtain the stated equality.
	
	Finally, from $\pi = (\id \times \nabla\psi)_\# \mu$ it follows that 
	\[
	\nu(B) 
	= \pi(\Rd \times B) 
	= \mu \bigl( (\id \times \nabla\psi)^{-1}(\Rd \times B) \bigr) 
	= \mu \bigl( (\nabla\psi)^{-1}(B) \bigr) 
	= (\nabla\psi_\#\mu)(B)
	\] 
	for every Borel set $B \subset \Rd$.
\end{proof}

\section{Existence and stability of optimal coupling measures}
\label{sec:stability} 

\subsection{Convergence of infinite measures}

To study the stability of coupling measures and transport plans between measures in $\Mz(\Rd)$, a notion of convergence is required that can handle sequences of measures with infinite mass. To this end, \citet{HultLindskog2006} consider the set $\mathcal{C}_{0}(\Rd)$ of continuous, bounded, real functions $f$ on $\reals^{d}$ for which there exists $r > 0$ such that $f(x) = 0$ whenever $\norm{x} \le r$. They endow $\Mz(\Rd)$ with the topology generated by open sets of the form
\[
\left\{
\nu \in \Mz(\reals^{d}) :
\lvert\textstyle\int f_{i} \,\diff\nu - \textstyle\int f_{i} \,\diff\mu\rvert 
< \eps,\, 
i=1,\ldots,k
\right\} 
\]
for $\mu \in \Mz(\Rd)$, positive integer $k$, functions $f_{1},\ldots,f_{k}$ in $\mathcal{C}_{0}(\Rd)$, and $\eps > 0$. The resulting topology on $\Mz(\Rd)$ is metrizable \citep[Theorem~2.3]{HultLindskog2006}. Convergence in $\Mz(\Rd)$ is denoted by $\vto$ and, in contrast to weak convergence, allows mass to accumulate near the origin (but not at infinity). A convenient criterion for $\nu_n \vto \nu$ is that $\int f \diff \nu_n \to \int f \diff \nu$ for all $f \in \mathcal{C}_0(\Rd)$. In many respects, $\Mz$-convergence is similar to weak convergence of finite Borel measures, with versions of the  Portmanteau lemma, the continuous mapping theorem, and criteria for relative compactness, see for instance \citep{HultLindskog2006, lindskog+r+r:2014}. 

Although all results will be formulated in terms of sequences of measures converging in $\Mz(\Rd)$, they also apply to weakly convergent sequences of probability measures. Clearly, the restriction of a probability measure on $\Rd$ to the complement of $\{0\}$ belongs to $\Mz(\Rd)$. Moreover, weak convergence of probability measures, denoted by $\wto$, is equivalent to convergence in $\Mz(\Rd)$ of their restrictions to $\Rd \setminus \{0\}$. 

\begin{lemma}
	\label{lem:weakMz}
	For a sequence of measures $(\mu_n)_n$ in $\Prob(\reals^d)$, we have $\mu_n \wto \mu \in \Prob(\reals^d)$ if and only if $\mu_n(\point \setminus \{0\}) \vto \mu' \in \Mz(\reals^d)$. Furthermore, $\mu'= \mu(\point \setminus \{0\})$.
\end{lemma}

\begin{proof}
	On the one hand, weak convergence implies $\Mz$-convergence, since the former involves a larger class of test functions of which the integrals should converge, namely all continuous, bounded real functions on $\Rd$, which includes $\mathcal{C}_0(\Rd)$; $\mu'= \mu(\point \setminus \{0\})$ follows immediately. 
	
	On the other hand, $\mu_n(\point \setminus \{0\}) \vto \mu' \in \Mz(\reals^d)$ implies tightness of the sequence of probability measures $(\mu_n)_n$, so by Prokhorov's theorem, every subsequence has a further subsequence converging weakly to some $\mu \in \Prob(\reals^d)$. But then $\mu_n(\point \setminus \{0\}) \vto \mu(\point \setminus \{0\})$ along this subsequence, so $\mu(\point \setminus \{0\})= \mu'$. This determines $\mu$ uniquely, as it is a probability measure. Therefore,  $\mu_n \wto \mu$ as  $n \to \infty$ along the full sequence. 
\end{proof}

\subsection{Existence and stability}


Theorem~\ref{thm:unique} treated the uniqueness of the gradient of a convex function pushing one measure in $\Mz(\Rd)$ forward to another one. Here, we will treat both the existence of such transport plans as well as their stability with respect to convergence of the underlying measures in $\Mz(\Rd)$. The proof of the existence in Theorem~\ref{thm:existence:Mz} below is based on an approximation argument involving sequences of finite measures $\mu_n$ and $\nu_n$, for which \citep[Theorem~6]{McCann} already guarantees the existence of coupling measures $\pi_n$ with cyclically monotone support, contained in the graph of the subdifferential of a closed convex function $\psi_n$. To make the argument work, we need to be assured of the stability of the construction with respect to convergence in $\Mz(\Rd)$, and this is guaranteed by Theorem~\ref{thm:basic_stability} below, which also applies to possibly infinite measures $\mu_n$ and $\nu_n$.
Using relative compactness arguments, we show convergence of the coupling measures $\pi_n$ and the subdifferentials $\partial \psi_n$ along subsequences. Different subsequences may have different limits, but under a weak smoothness assumption on $\bar{\mu}$, the limits are essentially unique, yielding convergence along the full sequence. For the subdifferentials, we consider graphical convergence as in Appendix~\ref{sec:gconv}, while for the coupling measures, we consider convergence in $\Mz(\Rd \times \Rd)$. 
Recall that $d_H$ denotes Hausdorff distance, see \eqref{eq:hausdorff}.

	
\begin{theorem}[Stability]
	\label{thm:basic_stability}
	Let $\mu_n, \nu_n, \bar{\mu}, \bar{\nu} \in \Mz(\Rd)$ be such that $\mu_n(\Rd) = \nu_n(\Rd) \in (0, +\infty]$ and $\bar{\mu}(\Rd) = \bar{\nu}(\Rd) \in (0, +\infty]$. Suppose that
	\begin{equation}
	\label{eq:marg_conv}
	  \mu_{n} \vto \bar{\mu},
	  \qquad \nu_{n} \vto \bar{\nu},
	  \qquad n \to \infty.
	\end{equation}
	Let $(\psi_{n})_n$ be a sequence of closed convex functions such that for each $n \in \mathbb{N}$, there exists $\pi_{n} \in \varPi(\mu_{n}, \nu_{n})$ such that $\spt(\pi_{n}) \subset \gph(\partial\psi_{n})$.
	\begin{enumerate}[(a)]
		\item
		Every subsequence contains a further subsequence for which there exists $\bar{\pi} \in \varPi(\bar{\mu}, \bar{\nu})$ and a closed convex function $\bpsi$ on $\Rd$ such that $\supp \bar{\pi} \subset \gph(\partial\bpsi)$ and, along the subsequence, $\pi_n \vto \bar{\pi}$ and $\partial \psi_n \gto \partial \bpsi$.
		
		\item
		If, in addition, $\bar{\mu}$ vanishes on sets of Hausdorff dimension at most $d-1$, then $\pi_n$ converges in $\Mz(\Rd \times \Rd)$ along the full sequence to the unique measure $\bar{\pi} \in \varPi(\bar{\mu}, \bar{\nu})$ with cyclically monotone support. We have $\bar{\pi} = (\id \times \nabla \bpsi)_\# \bar{\mu}$ and $\bar{\nu} = \nabla\bpsi_\# \bar{\mu}$, with $\nabla \bpsi$ equal to any of the subsequence limits in~(a), which are determined uniquely 
		$\bar{\mu}$-almost everywhere\footnote{In the sense that any two functions $\bpsi$ that can appear in (a) have gradients that are equal $\bar{\mu}$-almost everywhere, and are equal on the intersection of their domains with $\supp \bar\mu$ as in Theorem~\ref{thm:unique}.}. 
		\item
		If in addition,  $V := \intr(\supp(\bar{\mu}))$ is nonempty, then the subdifferentials $\partial\bpsi$ of the functions $\bpsi$ in (a) all coincide on $V$ and for every compact $K \subset V$ and every $\eps> 0$, there exists $n_{\eps,K}\in\NN$ such that for all $n\geq n_{\eps,K}$ and all $A\subset K$, 
	\begin{equation}
        \begin{split}
        \label{eq:include}
        \partial\bpsi(A) &\subset \partial\psi_{n}(A+\eps\ball)+\eps\ball,\\
        \partial\psi_{n}(A) &\subset \partial\bpsi(A+\eps\ball)+\eps\ball. 
        \end{split}
        \end{equation}
        The sets $\partial\bpsi(K)$ and $\partial\psi_{n}(K)$ are compact (the latter for sufficiently large $n$),  and 
        \begin{equation}
        \lim_{\eps\downarrow0}\limsup_{n\to\infty}
        d_{H}\left(\partial\psi_{n}(K+\eps\ball),\partial\bpsi(K)\right)=0.
        \label{eq:Hausdorff}
        \end{equation}
	\end{enumerate}
\end{theorem}

\begin{remark}
In fact, under the assumptions in (c), $\partial\psi_{n}$ converges graphically to $\partial\bpsi$ relative to $V$ as $n \to \infty$ as defined on p. 168 of \citet{Rockafellar-Wets}, and 
this in turn implies the locally uniform convergence results for the subdifferentials in (c). For brevity, we prove the latter results more directly from well-known properties of graphical convergence and the coincidence on $V$ of the graphical limits of $\partial\psi_{n}$ over subsequences as in (a).
\end{remark}

\begin{theorem}[Existence]
	\label{thm:existence:Mz}
	Let $\mu, \nu \in \Mz(\Rd)$ have equal, non-zero mass. There exists a coupling measure $\pi \in \varPi(\mu, \nu)$ with cyclically monotone support contained in the graph of the subdifferential of a closed convex function $\psi$ on $\Rd$. If $\mu$ vanishes on all sets of Hausdorff dimension at most $d-1$, then $\psi$ is differentiable $\mu$-almost everywhere, $\nabla\psi_\# \mu = \nu$, and the map $\nabla\psi$ is uniquely determined $\mu$-almost everywhere.
\end{theorem}

If $\mu$ and $\nu$ in Theorem~\ref{thm:existence:Mz} have infinite mass, then by Lemma~\ref{lem:Psi-finite}, the convex potential $\psi$ attains its minimum at the origin.

The following corollary extends Theorem~\ref{thm:basic_stability} to the case of weakly converging probability measures. Stability of transport plans between probability measures has been studied before \citep[see for instance][Theorem~2.8]{dBL:2019}, but the uniform convergence statements under~(c) seem to be new.

\begin{corollary}
	\label{cor:basic_stability:weak}
	Theorem~\ref{thm:basic_stability} continues to hold if $\mu_n, \nu_n, \bar{\mu}, \bar{\nu}$ are probability measures on $\Rd$ and if $\Mz$-convergence in \eqref{eq:marg_conv} is replaced by weak convergence.
\end{corollary}

\subsection{Proofs}

\begin{proof}[Proof of Theorem~\ref{thm:basic_stability}]
	(a) The sequences $(\mu_n)_n$ and $(\nu_n)_n$ converge in $\Mz(\Rd)$ and are thus relatively compact. By Lemma~\ref{lem:rlcmpct}, the sequence $(\pi_{n})_{n}$ is then relatively compact in $\Mz(\Rd \times \Rd)$. Every subsequence of $(\pi_{n})_n$ has therefore a further subsequence, with indices in $N \subset \NN$, say, that converges in $\Mz(\reals^{d} \times \reals^{d})$ to some measure $\bar{\pi}$, possibly depending on $N$. 
	
	By Lemma~\ref{lem:liminfsupp}, $\supp(\bar{\pi})$ is a subset of the Painlev\'e--Kuratowski inner limit of $\supp(\pi_{n})$ as $n \to \infty$ in $N$. Since $\supp(\pi_{n})\subset\graph(\partial\psi_{n})$, also
	\begin{equation}
	\label{eq:supp_incl_liminf}
	\supp(\bar{\pi})
	\subset
	\liminf_{n\to\infty, n \in N} \graph(\partial\psi_{n}),
	\end{equation}
	so that $\partial\psi_{n}$ cannot escape to the horizon as $n \to \infty$ in $N$. 
	Therefore, there is a further subsequence with indices in $M \subset N$ such that $\partial\psi_{n}$ converges graphically to some multivalued map $\bar{T}$ with nonempty domain as $n \to \infty$ in $M$ \citep[Theorem~5.36]{Rockafellar-Wets}. 
	By Lemma~\ref{lem:grphT}, $\bar{T}$ must be equal to the subdifferential of some closed convex function $\bpsi$.
	As $\graph(\partial\bpsi)$ is the Painlev\'e--Kuratowski limit of $\graph(\partial\psi_n)$ as $n \to \infty$ in $M$, it is also equal to the inner limit of those sets, which contains the inner limit of $\graph(\partial\psi_{n})$ as $n \to \infty$ in the larger subsequence $N$. 
	By \eqref{eq:supp_incl_liminf}, we find $\supp(\bar{\pi}) \subset \graph(\partial\bpsi)$. 
	
	(b) Let $(\bar{\pi}, \partial\bpsi)$ be a possible subsequence limit in (a). By Lemma~\ref{lem:supp-proj} and by~\eqref{eq:domains}, $\supp \bar{\mu} \subset \cl( \dom \partial \bpsi ) = \cl( \dom \bpsi )$. Since  $\bpsi$ is differentiable everywhere on $\cl( \dom \bpsi )$ except for a set of Hausdorff dimension at most $d-1$ \citep{AndersonKlee}, $\bpsi$ is differentiable and finite $\bar{\mu}$-almost everywhere. Measures in $\Mz(\Rd)$ are obviously $\sigma$-finite. Proposition~\ref{prop:rprgmm} in combination with (a) then implies that $\bar{\pi} = (\id \times \nabla \bpsi)_\# \bar{\mu}$ and $\bar{\nu} = \nabla\bpsi_\# \bar{\mu}$, with $\nabla \bpsi$ defined $\bar{\mu}$-almost everywhere.  By Theorem~\ref{thm:unique}, the map $\nabla \bpsi$ is unique $\bar{\mu}$-almost everywhere. Hence, all subsequence limits $\bar{\pi}$ coincide, implying that $\pi_n$ converges to $\bar{\pi}$. Similarly, any coupling measure of $\bar{\mu}$ and $\bar{\nu}$ with cyclically monotone support is of the form $(\id \times \nabla\bpsi)_\# \bar{\mu}$ for some closed convex function $\bpsi$ (Proposition~\ref{prop:rprgmm}), and since then $\bar{\nu} = \nabla\bpsi_\# \bar{\mu}$, Theorem~\ref{thm:unique} uniquely determines $\nabla\bpsi$ up to $\bar{\mu}$-null sets, thereby fixing the coupling measure too.


		(c)  For any pair of closed convex functions $\bpsi_a$ and $\bpsi_b$ whose subdifferentials can occur as graphical limits of (possibly different) subsequences in (a), we can find a $\bar{\mu}$-null set $S$ such that $\bpsi_a$ and $\bpsi_b$ are differentiable on $V \setminus S$ and their gradients coincide on that set. As a consequence, also $\partial \bpsi_a(x) = \{ \nabla \bpsi_a(x) \} = \{ \nabla \bpsi_b(x) \} = \partial \bpsi_b(x)$ for all $x \in V \setminus S$. For every open ball $B$ in $V$, the set $B \setminus S$ is dense in $B$ (Lemma~\ref{lem:US}). Therefore, by Corollary~1.5 of \citet{A-A}, $\partial\bpsi_{a}(x) = \partial\bpsi_{b}(x)$ for every $x\in B$ for every open ball $B$ in $V$ and therefore, for every $x\in V$. Therefore, all functions $\bpsi$ occurring in (a) have the same subdifferential $\partial\bpsi(x)$ and thus the same gradient $\nabla\bpsi(x)$ for all $x \in V$. 
	
		The remaining assertions hold along every graphically convergent subsequence $\partial\psi_n$, as they apply to any sequence of maximal monotone maps $T_n$ converging graphically to the maximal monotone map $T$ if $K \subset V$ for some open $V \subset \dom T$; see Proposition~\ref{prop:M-K}. The uniqueness of the limit $\partial\bpsi$ in $V$ then implies  the other claims: take \eqref{eq:include}, for example, and suppose that it does not hold. Then for some $\eps_0>0$, some $A \subset K$ and some infinite $N \subset \NN$, 
$\partial\bpsi(A) \not\subset \partial\psi_{n}(A+\eps_0\ball)+\eps_0\ball$ 
for all $n \in N$, or 
$\partial\psi_{n}(A) \not\subset \partial\bpsi(A+\eps_0\ball)+\eps_0\ball$
for all $n \in N$. Take any infinite subset of $N$ on which $\partial\psi_n \gto \partial\bpsi_A$ 
for some convex $\bpsi_A$ (see (a)); note that $\partial\bpsi_A(x)= \partial\bpsi(x)$ for all $x \in V$.  Then \eqref{eq:include} must hold with $\eps= \eps_0$ for $n$ large enough along this subsequence, which produces a contradiction. The proof of \eqref{eq:Hausdorff} is similar. 
\end{proof}

\begin{proof}[Proof of Theorem~\ref{thm:existence:Mz}]
	If the common value of $\mu(\Rd)$ and $\nu(\Rd)$ is finite, we can normalize $\mu$ and $\nu$ to become probability measures and apply McCann's results in \citep{McCann}. So assume that $\mu(\Rd) = \nu(\Rd) = +\infty$.
	
	For $n \in \NN$, let $\nu_n$ be the restriction of $\nu$ to the complement of $n^{-1} \ball$. Then $a_n := \nu_n(\Rd) = \nu(\Rd \setminus n^{-1}\ball)$ is finite but grows to infinity as $n \to \infty$. Further, $\nu_n \vto \nu$ as $n \to \infty$, since for every $f \in \mathcal{C}_0(\Rd)$, we have $\nu_n(f) = \nu(f)$ for every $n \in \NN$ that is sufficiently large such that $f$ vanishes on $n^{-1}\ball$.
	
	Let $\eps_n = \inf \{ \eps > 0 : \mu(\Rd \setminus \eps\ball) \le a_n \}$, for sufficiently large $n$ such that the set in the definition of the infimum is not empty. Necessarily $\mu(\Rd \setminus \eps_n\ball) \le a_n$ and $\mu(\Rd \setminus \eps_n\ball) \to +\infty$ and thus $\eps_n \to 0$ as $n \to \infty$. Put $m_n = a_n - \mu(\Rd \setminus \eps_n\ball)$, let $\kappa_n$ denote the Lebesgue-uniform distribution on $\eps_n\ball$ and let $\mu_n$ be the sum of $\mu(\point \setminus \eps_n\ball)$ and $m_n \kappa_n$. Then $\mu_n(\Rd) = a_n$ by construction and $\mu_n \vto \mu$ as $n \to \infty$, again because for every $f \in \mathcal{C}_0$ we have $\mu_n(f) = \mu(f)$ for all sufficiently large $n$.
	
	Apply \citet{McCann}, Theorem~6]  to the probability measures $\tilde{\mu}_n = a_n^{-1} \mu_n$ and $\tilde{\nu}_n = a_n^{-1} \nu_n$ to find $\tilde{\pi}_n \in \varPi(\tilde{\mu}_n, \tilde{\nu}_n)$ with cyclically monotone support. Then $\pi_n = a_n \tilde{\pi}_n$ belongs to $\varPi(\mu_n, \nu_n)$ and has the same, cyclically monotone support as $\tilde{\pi}_n$. By Rockafellar's theorem, there exists a closed convex function $\psi_n$ on $\Rd$ such that $\spt \pi_n$ is contained in $\gph \partial \psi_n$. Apply Theorem~\ref{thm:basic_stability}(a) to extract a subsequence along which $\pi_n \vto \pi$ and $\partial \psi_n \gto \partial \psi$ as $n \to \infty$, respectively, for some $\pi \in \varPi(\mu, \nu)$ and some closed convex function $\psi$ on $\Rd$ with the property that $\spt \pi \subset \gph(\partial \psi)$.
	
	If $\mu$ vanishes on sets with Hausdorff dimension not larger than $d-1$, we can apply Theorem~\ref{thm:unique} and Proposition~\ref{prop:rprgmm} to find the representation in terms of $\nabla\psi$ and the uniqueness $\mu$-a.e.\ of the latter. 
\end{proof}

\begin{proof}[Proof of Corollary~\ref{cor:basic_stability:weak}]
	   Lemma~\ref{lem:weakMz} provides a close connection between weak convergence and $\Mz$--convergence of probability measures. The remaining task is to ensure that the restriction of the measures and their couplings to the complement of the origin does not pose problems.

	We can always find a point $x_0 \in \Rd$ which is not an atom of any of the probability measures $\mu_n, \nu_n, \bar{\mu}$ and $\bar{\nu}$. Therefore, replacing these measures by their push-forwards under the translation $x \mapsto x - x_0$, all coincide with their restrictions to $\Rd \setminus \{0\}$. Thus, by Lemma~\ref{lem:weakMz}, weak convergence is equivalent to $\Mz$--convergence in \eqref{eq:marg_conv} to the same limits, and Theorem~\ref{thm:basic_stability} applies. 
	Since these translated measures do not charge the origin in $\Rd$, couplings between them do not charge the origin in $\Rd \times \Rd$, and by Lemma~\ref{lem:weakMz}, $\Mz$--convergence of a sequence of these couplings implies its weak convergence to the same limit, which is a probability measure. Therefore, all claims of Theorem~\ref{thm:basic_stability} with weak convergence instead of $\Mz$--convergence are satisfied for the translated measures $\mu_n, \nu_n, \bar{\mu}$ and $\bar{\nu}$. 
	Since the validity of these claims is unaffected by a translation of all measures and maps involved, the inverse translation  $x \mapsto x + x_0$  confirms the claims for the original measures. 
\end{proof}

\section{Tails of transport plans between regularly varying distributions}
\label{sec:RV}

\subsection{Regularly varying probability measures}
\label{subsec:RV}

In probability, regular variation underpins asymptotic theory for sample maxima and sums \citep{Bingham, Laurens__boek, Resnick_book} as it provides a coherent framework to describe the behaviour of functions at infinity and thus for the tails of random variables and vectors. A Borel measurable function $f$ defined in a neighbourhood of infinity and taking positive values is regularly varying with index $\tau \in \reals$ if $f(\lambda r)/f(r) \to \lambda^\tau$ as $r \to \infty$ for all $\lambda \in (0, \infty)$. A random variable $R$ is said to have a regularly varying upper tail with index $\alpha \in (0, \infty)$ if the function $r \mapsto \prob(R > r)$ is regularly varying with index $-\alpha$. The function $t \mapsto b(t) = Q(1 - 1/t)$, with $Q$ the quantile function of $R$, is then regularly varying at infinity with index $1/\alpha$, and we have $t \, \prob[R/b(t) > \lambda] \to \lambda^{-\alpha}$ as $t \to \infty$ for all $\lambda \in (0, \infty)$. The latter statement means that, in the space $\Mz(\reals)$, we have $t \, \prob[R/b(t) \in \point] \vto \bar{\nu}$ as $t \to \infty$ where $\bar{\nu}$ is concentrated on $(0, \infty)$ and is determined by $\bar{\nu}((\lambda,\infty)) = \lambda^{-\alpha}$ for all $\lambda \in (0, \infty)$.

More generally, a probability measure $\nu \in \Prob(\Rd)$ is regularly varying if there exists a Borel measurable function $b$ defined in a neighbourhood of infinity and taking positive values such that
\begin{equation}
\label{eq:MRV_cont}
t \, \nu(b(t) \point) \vto \bar{\nu},
\qquad t \to \infty
\end{equation}
for some non-zero $\bar{\nu}\in\Mz(\reals^{d})$. The function $b$ must be regularly varying with index $1/\alpha \in (0, \infty)$, say, and the limit measure $\bar{\nu}$ must be homogeneous:
\begin{equation*}
	\bar{\nu}(\lambda^{-1/\alpha} \point ) = \lambda \, \bar{\nu}(\point),
	\qquad \lambda > 0.
\end{equation*}
As a consequence, its support is a (closed) multiplicative cone, that is, $x \in \spt \bar{\nu}$ and $\lambda \in (0, \infty)$ imply $\lambda x \in \spt \bar{\nu}$. 
We call $\alpha \in (0, \infty)$ the index of regular variation of $\nu$.

\subsection{Tail limits of the transport map}
\label{sec:RV:tqc}

Let $\mu \in \Prob(\Rd)$ and $\nu \in \Prob(\Rd)$  be regularly varying. 
By Theorem 6 of \citet{McCann} and Rockafellar's Theorem \ref{thm:Rock}, there exists a coupling measure $\pi\in\varPi(\mu,\nu)$ with cyclically monotone support contained in the graph of the subdifferential $\partial\psi$ of some closed convex function $\psi$ on $\reals^{d}$ (see Section \ref{sec:OT}). The asymptotic behaviour of 
$\partial\psi(x)$ when $\norm{x}$ tends to infinity is described by the following application of Theorem~\ref{thm:basic_stability}. For notational convenience, write $1_d = (1, \ldots, 1) \in \Rd$. Then $\diag(a 1_d, b 1_d)$ is the $(2d) \times (2d)$ diagonal matrix with diagonal $(a, \ldots, a, b, \ldots, b)$, with $d$ repetitions of $a$ and $d$ repetitions of $b$. 
If $A = \diag(a_1, \ldots, a_d)$ is a diagonal matrix and $\lambda > 0$ is a positive scalar, then we write $\lambda^A = \exp(A \log(\lambda)) = \diag(\lambda^{a_1}, \ldots, \lambda^{a_d})$ with $\exp(\point)$ the matrix exponential.

\begin{theorem}
\label{thm:RV}
    	Let $\mu$ and $\nu$ in $\Prob(\Rd)$ be regularly varying with auxiliary functions $b_1$ and $b_2$, indices $\alpha_1 > 0$ and $\alpha_2 > 0$, and limit measures $\bar{\mu}$ and $\bar{\nu}$ in $\Mz(\Rd)$. Let $\psi$ be a closed convex function such that the graph of $\partial\psi$ contains the cyclically monotone support of some $\pi\in\varPi(\mu,\nu)$.     	
    	\begin{enumerate}[(a)]
    		\item
    		Every sequence of positive numbers $t_n \to \infty$ contains a subsequence such that, as $n \to \infty$ along the subsequence, writing $B(t) = \diag(b_1(t) 1_d, b_2(t) 1_d)$, we have
		\begin{equation}
		\label{eq:RV:conv}
			\left.
			\begin{split}
			t_n \, \pi(B(t_n) \point) &\vto \bar{\pi} \qquad \text{in $\Mz(\Rd \times \Rd)$}\\
			(b_2(t_n))^{-1} \partial \psi(b_1(t_n) \point) &\gto \partial \bpsi
			\end{split}
			\right\}
		\end{equation}
		for some $\bar{\pi} \in \varPi(\bar{\mu}, \bar{\nu})$ and some closed convex function $\bpsi$ satisfying $\spt \bar{\pi} \subset \gph(\partial \bpsi)$ and $\bpsi(0) < \infty$.
	
		\item
		If the limiting measure $\bar{\mu}$ vanishes on sets of Hausdorff dimension not larger than $d-1$, then
		the coupling $\pi$ satisfies
    		\begin{equation}
    		\label{eq:pit2barpi}
    		t \, \pi(B(t) \point) \vto \bar{\pi}
    		\end{equation}
    		 in $\Mz(\Rd \times \Rd)$ as $n \to \infty$, 
		 where $\bar{\pi} = (\id \times \nabla\bpsi)_\# \bar{\mu}$ is the unique coupling measure 
		 between $\bar{\mu}$ and $\bar{\nu}$ having cyclically monotone support and satisfying
    		 \begin{equation}
		 \label{eq:pibarhomo}
    		 \bar{\pi}(\lambda^{-E}\point) = \lambda \bar{\pi},\qquad \lambda > 0
    		 \end{equation}
		 with $E = \diag(\alpha_{1}^{-1} 1_d, \alpha_{2}^{-1} 1_d)$.
		The gradient $\nabla\bpsi$ is determined 
		uniquely $\bar{\mu}$--almost everywhere and 
		satisfies $(\nabla\bpsi)_\#\bar{\mu} = \bar{\nu}$ and
	        \begin{equation}
		\label{eq:homog:1}
    		\nabla\bpsi(\lambda x) = \lambda^{\alpha_1/\alpha_2} \,\nabla\bpsi(x), \qquad \lambda > 0
    		\end{equation}
		$\bar\mu$-almost everywhere. 
	         Moreover, the function $\bpsi$ 
	         may be modified such that $\bpsi(0)= 0$; then it is determined uniquely $\bar{\mu}$--almost everywhere and satisfies
	         \begin{equation}
                 \label{eq:bpsi:homog}
    		\bpsi(\lambda x) = \lambda^{\alpha_1/\alpha_2+1} \, \bpsi(x),
    		\qquad \lambda \ge 0
    		\end{equation}
		$\bar\mu$-almost everywhere.

		  \item
		  If $V:= \intr(\supp\bar\mu)$ is nonempty, then $\partial\bpsi$ is unique on $V$ and satisfies 
		  $\partial\bpsi(\lambda x) = \lambda^{\alpha_1/\alpha_2} \,\partial\bpsi(x)$ for all $\lambda>0$ and $x \in V$.
		  Furthermore, for every positive sequence $t_n$ tending to infinity and $\psi_n:= (b_2(t_n) b_1(t_n))^{-1} \psi(b_1(t_n)
		  \point)$, the maps  $\partial\psi_n$ converge locally uniformly to $\partial\bpsi$ on $V$ in the sense of Theorem
		  \ref{thm:basic_stability}(c).
		  
    	\end{enumerate}
    \end{theorem}

\begin{remark}
In Theorem~\ref{thm:RV}(b), the measure $\pi$ is operator regularly varying in the sense of \citep{vMeerschaert}, provided that $\bar\pi$ is not supported on a subspace lower than dimension $2d$.
If $b_1 = b_2$ and thus $\alpha_1 = \alpha_2$, it is regularly varying as in Section~\ref{subsec:RV}.
\end{remark}	

\begin{proof}[Proof of Theorem~\ref{thm:RV}]

	(a) Choose a sequence $1 \le t_n \to \infty$ and consider the Borel measures $\mu_n = t_n \, \mu(b_1(t_n)\point)$ and $\nu_n = t_n \, \nu(b_2(t_n)\point)$, both with mass $t_n$. In view of Lemma~\ref{lem:scaling} below, the function $\psi_n = (b_2(t_n) b_1(t_n))^{-1} \psi(b_1(t_n) \point)$ is closed and convex, and the graph of $\partial\psi_n$ contains the support of the measure $\pi_n = t_n \, \pi(B(t_n) \point)$ that couples $\mu_n$ and $\nu_n$.
	By regular variation, $\mu_n \vto \bar{\mu}$ and $\nu_n \vto \bar{\nu}$ in $\Mz(\Rd)$ as $n \to \infty$, with $\bar{\mu}$ and $\bar{\nu}$ not depending on the sequence $t_n$. 
	By Theorem~\ref{thm:basic_stability}, a subsequence of $t_n$ (relabelled $t_n$) exists satisfying $\pi_n \vto \bar{\pi} \in \varPi(\bar{\mu}, \bar{\nu})$ in $\Mz(\Rd \times \Rd)$ and $\partial \psi_n \gto \partial\bpsi$ as $n \to \infty$, with $\bpsi$ a closed convex function satisfying $\supp\bar\pi \subset \gph \partial\bpsi$.
	
	Since $\bar{\mu}$ and $\bar{\nu}$ have infinite mass and since their coupling $\bar{\pi}$ is supported by $\gph(\partial \bpsi)$, Lemma~\ref{lem:Psi-finite} implies that $\bpsi(0)$ is finite.
	
	\smallskip
	
	(b) In view of the assumption on $\bar{\mu}$, Theorem~\ref{thm:basic_stability}(b) implies that there is a unique coupling measure $\bar{\pi} \in \varPi(\bar{\mu}, \bar{\nu})$ with cyclically monotone support. All limit measures in (a) must be equal to this unique measure, and thus we must have \eqref{eq:pit2barpi}. Since $B(t)$ is invertible for every $t>0$ and $B(\lambda t)^{-1} B(t) \to \lambda^{-E}$ as $t \to \infty$, the homogeneity~\eqref{eq:pibarhomo} follows from the theory of operator regular variation \citep[Proposition~6.1.2]{vMeerschaert}, as $\Mz$-convergence is equivalent to vague convergence in compactified and punctured Euclidean space, at least for measures that do not charge the artificially added points at infinity \citep[Proposition~4.4]{lindskog+r+r:2014}. Intuitively, for $\lambda \in (0, \infty)$, we have, on the one hand,
	\[
		\lambda t \, \pi(B(t) \point)
		\vto \lambda \bar{\pi},
		\qquad t \to \infty,
	\]
	and on the other hand,
	\begin{align*}
		\lambda t \, \pi(B(t) \point)
		&= \lambda t \, \pi(B(\lambda t) B(\lambda t)^{-1} B(t) \point) \\
		&\vto \bar{\pi}(\lambda^{-E} \point),
		\qquad t \to \infty.
	\end{align*}
	
	The homogeneity of $\bar{\pi}$ in \eqref{eq:pibarhomo} implies
	\begin{equation}
	\label{eq:spt_homog}
		\lambda^E \spt \bar{\pi} = \spt \bar{\pi},
	\end{equation}
	with $\lambda^E = \diag(\lambda^{1/\alpha_1} 1_d, \lambda^{1/\alpha_2} 1_d)$. 
	Let $T:\Rd \toto \Rd$ be the mapping satisfying $\gph T= \supp \bar\pi$. It is unique, and by \eqref{eq:spt_homog}, $y \in T(x)$ if and only if  $\lambda^{1/\alpha_2} y \in T(\lambda^{1/\alpha_1} x)$ by \eqref{eq:spt_homog}, so
	\begin{equation} 
	\label{eq:T_homog}
		T(\lambda x)= \lambda^{\alpha_1/\alpha_2} T(x), \qquad  \lambda>0, x \in \Rd,
	\end{equation}
	an identity in which one and hence both sets can be empty.
	
	By \eqref{eq:T_homog}, the set $\dom T$ is a multiplicative cone. Further, note that $\dom T= \proj_1(\supp \bar\pi)$, so $\bar\mu(\Rd \setminus \dom T)= 0$ and by Lemma~\ref{lem:supp-proj}, $\supp \bar\mu = \cl(\dom T)$.
	 As $\supp \bar\pi \subset \gph \partial\bpsi$, also $\dom T \subset \dom \partial \bpsi$.
	 	
	Replace $\bpsi$ by $\bpsi - \bpsi(0)$ to ensure that $\bpsi(0) = 0$ without changing the subdifferential or the gradient of $\bpsi$.
	Let $x \in \dom T$ and define $f(t) = \bpsi(t x)$ for $t \geq 0$. The convex function $f$ is  finite on $[0, \infty)$, as $ \dom T \subset \dom \partial\bpsi$ and $\dom T$ is a multiplicative cone.
	Its subdifferential  $\partial f(t)$ contains the set $\inpr{\partial \bpsi(tx), x} = \{ \inpr{y, x} : y \in \partial \bpsi(tx) \}$ (just check the inequality in the definition of a subdifferential); therefore, it also contains the set $\inpr{T(tx), x}$.
	As a consequence, $f$ has a left derivative $f'_{-} $ and a right derivative $f'_{+}$ satisfying 
	$f'_{-}(t) \leq \inf \inpr{T(tx), x} \leq \sup \inpr{T(tx), x} \leq f'_{+}(t)$ 
	at every $t>0$. 	
	As furthermore
	\begin{equation}
	\label{eq:f}
	f(t)
	= \int_{(0,t)} f'_{-}(s) \, \diff s
	= \int_{(0,t)} f'_{+}(s) \, \diff s, 
	\qquad t \in (0,\infty)
	\end{equation}	
	 (e.g. \cite{Rockafellar}, Corollary 24.2.1), Eq.~\eqref{eq:T_homog} implies 
	\begin{equation}
	\label{eq:bpsi}
	\bpsi(tx)
	= \int_{(0,t)} \sup \inpr{T(sx), x} \, \diff s
	= \sup \inpr{T(x), x} \int_{(0,t)}  s^{\alpha_1/\alpha_2} \, \diff s.
	\end{equation}
We find that $\bpsi(x)$ is uniquely determined by $T$ for $x \in \dom T$ and thus $\bar\mu$-almost everywhere and also that \eqref{eq:bpsi:homog} holds for such $x$. The set $\intr(\conv(\spt \bar\mu))$ is contained in $\intr(\dom \bpsi)$, since otherwise $\bar\mu$ would assign positive mass to an open half-space on which $\bpsi$ is infinite, in contradiction to the fact that $\bpsi$ is finite $\bar\mu$-almost everywhere. Since $\bpsi$ is continuous on the interior of its domain, we further deduce that $\bpsi(x)$ is uniquely determined by $T$ for $x \in \intr(\conv(\spt \bar\mu))$ and that \eqref{eq:bpsi:homog} holds for such $x$ too.

As in Lemma~\ref{lem:clconvhull}, let the function $\phi$ be the closure of the convex hull of the restriction of $\bpsi$ to $D = \dom T \cup \intr(\conv(\spt \bar\mu))$; its epigraph is $\epi \phi = \cl(\conv(\epi \bpsi \cap (D \times \reals)))$. By Lemma~\ref{lem:clconvhull}(c), $\partial \bpsi(x) \subset \partial \phi(x)$ for all $x \in D$ and therefore, 
\[ 
	\spt \bar\pi 
	\subset \gph (\partial \bar\psi) \cap (D \times \Rd) 
	\subset \gph( \partial \phi).
\]
By Proposition~\ref{prop:rprgmm}, we find that $\nabla \phi$ is defined $\bar{\mu}$-almost everywhere.

Furthermore, by Lemma~\ref{lem:clconvhull}(b), $\phi(x)= \bpsi(x)$ for $x \in D$, so as $D$ is a multiplicative cone and $\bpsi(\lambda x) = \lambda^{\alpha_1/\alpha_2+1} \bpsi(x)$ for all $x \in D$ and $\lambda \ge 0$,
\begin{equation}
\label{eq:phi_homog}
	\phi(\lambda x) = \lambda^{\alpha_1/\alpha_2+1} \phi(x), 
	\qquad \lambda > 0, \, x \in D.
\end{equation}
Take $x \in D$ and $\lambda > 0$. By Lemma~\ref{lem:clconvhull}(b)--(c), we have $v \in \partial\phi(x)$ if and only if $\phi(z) \ge \phi(x) + \inpr{v, z-x}$ for all $z \in D$.
Therefore, as $D$ is a multiplicative cone, $w \in \partial\phi(\lambda x)$ if and only if 
$\phi(\lambda z) \ge \phi(\lambda x) + \inpr{w, \lambda z - \lambda x}$ for all $z \in D$. Using \eqref{eq:phi_homog}, the latter inequality is equivalent to 
$\phi(z) \ge \phi(x) + \inpr{\lambda^{-\alpha_1/\alpha_2} w, z-x}$ for all $z \in D$. Combining (b) and (c) in Lemma~\ref{lem:clconvhull} once more, we find that $w \in \partial \phi(\lambda x)$ if and only if $\lambda^{-\alpha_1/\alpha_2} w \in \partial \phi(x)$ and thus
\[
	\partial\phi(\lambda x) = \lambda^{\alpha_1/\alpha_2} \partial\phi(x),
	\qquad \lambda>0, \, x \in D.
\]

As a consequence, $D \cap \dom\nabla\phi$ is a multiplicative cone and $\nabla\phi(\lambda x) = \lambda^{\alpha_1/\alpha_2} \nabla\phi(x)$ for all $\lambda>0$ and all $x \in D \cap \dom\nabla\phi$. In view of Lemma~\ref{lem:clconvhull}(e), Eq.~\eqref{eq:homog:1} must hold for all $x \in D \cap \dom\nabla\phi$ and therefore $\bar\mu$-almost everywhere. 

%
    \smallskip
	(c) Most statements follow directly from Theorem \ref{thm:basic_stability}(c); the homogeneity of $\partial\bpsi$ follows from the homogeneity of $\partial\phi$, as $\bpsi(x) = \phi(x)$ for all $x \in V$.
\end{proof}


\begin{lemma}[Scaling]
	\label{lem:scaling}
	Let $\mu, \nu \in \Prob(\Rd)$ and let $\pi \in \varPi(\mu, \nu)$ have support contained in the graph of $\partial \psi$ for some closed convex function $\psi$. For scalar $b_1 > 0$ and  $b_2 > 0$, consider the $2d \times 2d$ matrix $B=  \diag(b_{1} 1_d,b_{2} 1_d)$ and the measures $\mu_{b_1} = \mu(b_1 \point)$, $\nu_{b_2} = \nu(b_2 \point)$ and $\pi_{b_1,b_2} = \pi(B \point)$. 
	Then $\pi_{b_1,b_2} \in \varPi(\mu_{b_1}, \nu_{b_2})$ and its support $\spt \pi_{b_1,b_2} = B^{-1} \spt \pi$ is contained in the graph of the subdifferential of the closed convex function $\psi_{b_1,b_2} = (b_1 b_2)^{-1} \psi(b_1 \point)$, the subdifferential of which is $\partial \psi_{b_1,b_2}(x) = b_2^{-1} \partial \psi(b_1 x)$ for all $x$, with graph $\gph(\partial \psi_{b_1, b_2}) = B^{-1} \gph(\partial \psi)$.
\end{lemma}

\begin{proof}
	The function $\psi_{b_1,b_2}$ is closed and convex since this is true for $\psi$. 
	We have $(x, y) \in \gph \partial \psi_{b_1, b_2}$ if and only if
	\[
		\forall z \in \Rd, \qquad
		(b_1b_2)^{-1} \psi(b_1 z)
		\ge
		(b_1b_2)^{-1} \psi(b_1 x)
		+
		\inpr{z-x, y},
	\]
	which is equivalent to
	\[
		\forall s \in \Rd, \qquad
		\psi(s) \ge \psi(b_1 x) + \inpr{s - b_1x, b_2y}
	\]
	and thus to $(b_1x, b_2y) \in \gph \partial \psi$. We find that $\gph(\partial \psi_{b_1,b_2}) = B^{-1} \gph(\partial \psi)$. 
	From the definition, $\spt \pi_{b_1,b_2} = B^{-1} \spt \pi$. Therefore, the support of $\pi_{b_1,b_2}$ is included in the graph of $\partial \psi_{b_1,b_2}$.
\end{proof}

Recall that the convex hull of a function $g : \Rd \to \reals \cup \{\infty\}$ is the function $\conv g$ with epigraph $\epi(\conv g) = \conv(\epi g)$, the convex hull of the epigraph of $g$. It is the greatest convex function majorized by $g$ \citep[p.~36]{Rockafellar}. Recall also that the closure of a function $f$ is the function $\cl f$ whose epigraph is $\epi(\cl f) = \cl(\epi f)$; it is the same as lower semicontinuous hull of that function.

\begin{lemma}
	\label{lem:clconvhull}
	Let $\psi : \reals^d \to \reals \cup \{\infty\}$ be a closed convex function and let $D \subset \Rd$ be such that $D \cap \dom \psi \ne \varnothing$. Define the function $\phi$ by $\phi = \cl(\conv g)$, where $g(x) = \psi(x)$ if $x \in D$ and $g(x) = \infty$ if $x \in \Rd \setminus D$. Then $\phi$ is a closed convex function with the following properties:
	\begin{enumerate}[(a)]
		\item $\phi \ge \psi$;
		\item $\phi(x) = \psi(x)$ for $x \in D$;
		\item $\partial \phi(x) = \{ v \in \Rd : \forall z \in D, \, \psi(z) \ge \psi(x) + \inpr{z-x,v} \} \supset \partial \psi(x)$ for $x \in  D$;
		\item $\dom \nabla \phi \subset \dom \partial \psi$;
		\item $D \cap \dom \nabla \phi \subset \dom \nabla \psi$ and $\nabla \phi(x) = \nabla \psi(x)$ for all $x \in D \cap \dom \nabla \phi$.
	\end{enumerate}
\end{lemma}

\begin{proof}
	Since $\epi \phi = \cl(\epi (\conv g)) = \cl(\conv (\epi g))$, the epigraph of $\phi$ is closed and convex. Hence $\phi$ is a closed and convex function.  
		
	(a) As $\epi g = \epi \psi \cap (D \times \reals)$, we have $\epi \psi \supset \epi g$ and thus, since $\epi \psi$ is convex and closed, $\epi \psi \supset \conv (\epi g)$ and then $\epi \psi \supset \cl(\conv(\epi g)) = \epi \phi$. Hence $\psi \le \phi$.
	
	
	(b) Since $\epi \phi \supset \epi g$, we must have $\phi(x) \le g(x) = \psi(x)$ for $x \in D$, and thus $\phi(x) = \psi(x)$ for such $x$.
	

	(c) Let $x \in D$ and write $V = \{ v \in \Rd : \forall z \in D, \, \psi(z) \ge \psi(x) + \inpr{z-x,v} \}$.
	By definition of the subdifferential, $V \supset \partial \psi(x)$. Likewise, in view of (b), also $V \supset \partial \phi(x)$. We need to show that $V \subset \partial \phi(x)$. Let $v \in V$. By definition, 
	\[
		\epi g 
		= \epi \psi \cap (D \times \reals) 
		\subset \{ (z, \lambda) \in \Rd \times \reals : \lambda \ge \psi(x) + \inpr{z-x,v} \}.
	\]
	The set on the right-hand side is a closed half-space in $\reals^{d+1}$. Therefore, it also contains 
	$\cl(\conv(\epi g)) = \epi \phi$.
	But that means that $\phi(z) \ge \psi(x) + \inpr{z-x,v}$ for all $z \in \Rd$. As $\psi(x) = \phi(x)$, we find that $v \in \partial \phi(x)$. 
	Therefore, $V \subset \partial \phi(x)$ as required.
	
	(d) From $\psi \le \phi$ we infer $\dom \phi \subset \dom \psi$. \mdf~{Therefore, by \eqref{eq:domains},}
	\[
		\dom \nabla \phi
		\subset \intr (\dom \phi)
		\subset \intr (\dom \psi)
		= \dom \partial \psi.
	\]
	
	(e) Let $x \in D \cap \dom \nabla \phi$ and write $\nabla \phi(x) = y$. On the one hand, by (d), $x \in \dom \partial \psi$ and thus $\partial \psi(x) \ne \varnothing$. On the other hand, by (c), $\partial \psi(x) \subset \partial \phi(x) = \{y\}$. We conclude that $\partial \psi(x) = \{y\}$, so that $x \in \dom \nabla \psi$ and $\nabla \psi(x) = y = \nabla \phi(x)$.
\end{proof}

\appendix

\section{Infinite measures: supports and convergence}
\normalsize

In this section, we collect some results on possibly infinite Borel measures on Euclidean space, with particular attention to $\Mz$-convergence and to properties of the supports of such measures.

\begin{lemma}[Coupling in $\Mz(\Rd)$]
	\label{lem:rlcmpct}
	\begin{enumerate}[(a)]
		\item
		If $\mu, \nu \in \Mz(\Rd)$, then every $\pi \in \varPi(\mu, \nu)$ belongs to $\Mz(\Rd \times \Rd)$.
		
		\item
		If $M$ and $N$ are relative compact sets of measures in $\Mz(\Rd)$, then $\bigcup \{ \varPi(\mu, \nu) : \mu \in M, \nu \in N \}$ is relatively compact in $\Mz(\Rd \times \Rd)$. 
	\end{enumerate}
\end{lemma}

\begin{proof}
	(a) If $(x, y) \in \Rd \times \Rd$ is such that $\norm{(x, y)} \ge r$ for some $r > 0$, then necessarily $\norm{x} \ge r/2$ or $\norm{y} \ge r/2$. For $\mu, \nu \in \Mz(\Rd)$ and $\pi \in \varPi(\mu, \nu)$, we thus have
	\begin{equation}
	\label{eq:piMz}
	\pi \bigl( \{ (x, y) : \norm{(x, y)} \ge r \} \bigr)
	\leq \mu(\{ x : \norm{x} \ge r/2 \}) + \nu(\{ y : \norm{y} \ge r/2 \}),
	\end{equation}
	which is finite by assumption.
	
	(b) A necessary and sufficient criterion for relative compactness of a subset $K$ in $\Mz(\Rd)$ is $\sup_{\mu \in K} \mu(S_r) < \infty$ for all $r>0$ and $\lim_{r \to \infty} \sup_{\mu \in K} \mu(S_{r}) = 0$ \citep[Section~4.1]{HultLindskog2006}. 
	Apply this criterion, twice in $\Mz(\Rd)$ and once in $\Mz(\Rd \times \Rd)$, in combination with~\eqref{eq:piMz}, to conclude the proof.
\end{proof}

\begin{lemma}[From joint to marginal convergence]
	\label{lem:mrgnlMz}
	Let $\mu_n, \nu_n \in \Mz(\Rd)$ and let $\pi_n \in \varPi(\mu_n, \nu_n)$ for all $n$. If $\pi_n \vto \pi$  as $n \to \infty$, then also $\mu_n \vto \mu$ and $\nu_n \vto \nu$ as $n \to \infty$ and $\pi \in \varPi(\mu, \nu)$.
\end{lemma}

\begin{proof}
	This is a consequence of Theorem~2.5 in \citet{HultLindskog2005} applied to the projection maps $(x, y) \mapsto x$ and $(x, y) \mapsto y$ from $\Rd \times \Rd$ into $\Rd$, each of which is continuous and maps the origin in $\Rd \times \Rd$ to the origin in $\Rd$.
\end{proof}

The inner limit of a sequence of sets $A_n \subset \Rd$ is the set $\liminf_{n \to \infty} A_n$ of points $x \in \Rd$ every neighbourhood of which intersects all but finitely many $A_n$ (Section~\ref{sec:gconv:PK}). Note that this topological inner limit is in general larger than the set-theoretic one, defined as $\bigcup_{n \ge 1} \bigcap_{k \ge n} A_k$.

\begin{lemma}[Support of a limit measure]
	\label{lem:liminfsupp}
	If $\mu_n$ converges to $\mu$ in $\Mz(\Rd)$, then $\supp(\mu) \subset \liminf_{n \to \infty} \supp(\mu_n)$.
\end{lemma}

\begin{proof}
	Let $U$ be an open neighbourhood of $x \in \supp(\mu)$. Then $\mu(U) > 0$. By the Portmanteau theorem for $\Mz$ \citep[Theorem~2.4]{HultLindskog2006},
	\[
	\liminf_{n \to \infty} \mu_{n}(U) \ge \mu(U) > 0.
	\]
	Hence $U$ has a non-empty intersection with $\supp \mu_{n}$ for all sufficiently large $n$, for otherwise, we could find a subsequence such that $\mu_n(U) = 0$ along the subsequence. Since $U$ was arbitrary, we conclude that $x$ belongs to the inner limit of $\supp(\mu_n)$.
\end{proof}

\begin{lemma}[Support of a Borel measure]
	\label{lem:US}
	Let $\mu$ be a non-zero Borel measure on $\Rd$. If $U \subset \supp(\mu)$ is non-empty and open and if $S \subset \Rd$ is a Borel set with $\mu(S) = 0$, then $U \setminus S$ is dense in $U$.
\end{lemma}

\begin{proof}
	Let $x \in U \cap S$ and let $V$ be a neigbourhood of $x$. We need to show that $V$ intersects $U \setminus S$. Since $x \in \supp(\mu)$ and since $U \cap V$ is a neigbourhood of $x$, we have $\mu(U \cap V) > 0$. Since $\mu(S) = 0$, we find $\mu((U \cap V) \setminus S) > 0$ too. As a consequence, $(U \cap V) \setminus S = (U \setminus S) \cap V$ cannot be empty.
\end{proof}

Let $\proj_1$ be the projection map $(x, y) \mapsto x$ from $\Rd \times \Rd$ into $\Rd$.

\begin{lemma}[Support of a margin]
	\label{lem:supp-proj}
	If $\mu = (\proj_1)_\# \pi$ is the left marginal of the Borel measure $\pi$ on $\Rd \times \Rd$, then $\supp \mu = \cl( \proj_1 (\supp \pi))$. As a consequence, if $\supp \pi \subset \graph( T )$ for some multivalued map $T : \Rd \rightrightarrows \Rd$, then $\supp \mu \subset \cl(\dom T)$.
\end{lemma}

\begin{proof}
	Put $G = \cl( \proj_1 (\supp \pi) )$. To show that $\spt \mu \supset G$, it suffices to show that $\spt \mu \supset \proj_1(\spt \pi)$, as $\spt \mu$ is closed. But the latter inclusion follows easily: if $x \in \proj_1(\spt \pi)$, then there exists $y$ such that $(x, y) \in \spt \pi$. For any open neighbourhood $U$ of $x$, the set $U \times \Rd$ is an open neighbourhood of $(x, y)$, and thus $\mu(U) = \pi(U \times \Rd) > 0$. As a consequence, $x \in \spt \mu$.	
	
	Conversely, to show that $\spt \mu \subset G$, we need to show that $\mu(G^c) = 0$, as $G$ is closed. We have $\mu(G^c) = \pi(G^c \times \Rd)$ and
	\[
	G^c \times \Rd
	\subset (\proj_1 (\supp \pi))^c \times \Rd
	\subset (\supp \pi)^c,
	\]
	which is a $\pi$-null set.
	
	For the second statement, it suffices to notice that
	\begin{equation*}
	\proj_1 (\supp \pi) \subset \proj_1(\graph T) = \dom T. 
	\qedhere
	\end{equation*}
\end{proof}

\section{Graphical convergence of multivalued maps}
\label{sec:gconv}

Subdifferentials of convex functions are multivalued maps. To study the asymptotic properties of a sequence of such maps, we employ the notion of graphical convergence, which is in turn based on the theory of Painlev\'e--Kuratowski set convergence. We review the basic definitions and facts and we state some results that we use in the paper and that we have not been able to find in the literature. Proofs are collected at the end. Our exposition 
leans heavily on \citep[Chapters~4, 5 and 12]{Rockafellar-Wets} and \citep[Appendix~B]{Molchanov2005}.

\subsection{Painlev\'e--Kuratowski set convergence}
\label{sec:gconv:PK}

The inner limit of a sequence of sets $C_n \subset \Rd$ is the set $\liminf_{n \to \infty} C_n$ of points $x \in \Rd$ for which there exists a sequence of points $x_n \in C_n$ such that $x_n \to x$ as $n \to \infty$. The outer limit of the sequence $C_n$ is the set $\limsup_{n \to \infty} C_n$ of points $x \in \Rd$ for which there exists an infinite set $N \subset \mathbb{N}$ such that $x_n \in C_n$ for all $n \in N$ and such that $x_n \to x$ as $n \to \infty$ in $N$. Equivalently, the inner limit is the set of all points $x$ every neighbourhood of which intersects all but a finite number of sets $C_n$, while the outer limit is the set of points $x$ every neighbourhood of which intersects an infinite number of sets $C_n$. The inner and outer limits are both closed, and the inner limit is contained in the outer limit.

A sequence of sets $C_n \subset \Rd$ is said to converge to $C \subset \Rd$ (written as $C_n \to C$) if the inner and outer limits agree and are equal to $C$; this is Painlev\'e--Kuratowski set convergence. When restricted to the space of closed subsets of $\Rd$, it is equivalent to convergence in the Fell hit-and-miss topology \citep[Theorem~B.6]{Molchanov2005}. From the definitions above, it follows immediately that if $C_n \to C$ and $\lambda_n \to \lambda$ for some sequence of positive numbers $\lambda_n$ with positive limit $\lambda$, then $\lambda_n C_n \to \lambda C$, generalizing Proposition~D.5 in \citep{Molchanov2005}.

\subsection{Graphical convergence}

The graphical inner limit of a sequence of maps $S_n : \reals^k \toto \reals^\ell$ is the map $\gliminf_{n \to \infty} S_n$ whose graph is equal to the inner limit of the sequence of graphs of $S_n$. Similarly, the graphical outer limit of the sequence $S_n$ is the map $\glimsup_{n \to \infty} S_n$ whose graph is equal to the outer limit of the graphs of $S_n$. 
The sequence $S_n$ converges graphically to $S : \reals^k \toto \reals^\ell$, notation $S_n \gto S$, if the graphical inner and outer limits coincide and are equal to the graph of $S$. In other words, we have $S_n \gto S$ as $n \to \infty$ if the inner and outer limits of $\gph(S_n)$ in $\reals^k \times \reals^\ell$ are equal to $\gph(S)$.

We will apply the concept of graphical convergence to subdifferentials of convex functions. A useful fact is that the graphical limit of a sequence of subdifferentials of closed convex functions is again the subdifferential of a closed convex function.

\begin{lemma}[Graphical limits of subdifferentials]
	\label{lem:grphT}
	Let $\psi_n$ be a sequence of closed convex functions on $\Rd$ and suppose that $\partial \psi_n$ converges graphically to some map $T : \Rd \toto \Rd$ with non-empty domain. Then $T = \partial \psi$ for some closed convex function $\psi$ on $\Rd$.
\end{lemma}

\begin{proof}[Proof of Lemma~\ref{lem:grphT}]
	Since $\graph(\partial\psi_{n})$ is cyclically monotone, its limit $\graph(T)$ is cyclically monotone too \citep[proof of Theorem~5.20]{Villani}. Furthermore, since $\graph(\partial\psi_{n})$ is maximal monotone \citep[Corollary~31.5.2]{Rockafellar}, its limit is maximal monotone too \citep[Theorem~12.32]{Rockafellar-Wets}. Because cyclic monotonicity implies monotonicity, the graph of $T$ must be maximal cyclically monotone. Rockafellar's Theorem~\ref{thm:Rock} then implies that $T = \partial\psi$ for some closed convex function $\psi$.
\end{proof}

To show the existence of graphically converging subsequences of a sequence of maps, there exists a simple criterion. A sequence of maps $S_n : \reals^k \toto \reals^\ell$ is said to escape to the horizon if for every bounded set $B \subset \reals^k \times \reals^\ell$, there exists an infinite set $N \subset \mathbb{N}$ such that $\gph(S_n)$ does not intersect $B$ for all $n \in N$. In other words, for all bounded sets $C \subset \reals^k$ and $D \subset \reals^\ell$, there exists an infinite set $N \subset \mathbb{N}$ such that $S_n(x) \cap D = \varnothing$ for all $x \in C$ and $n \in N$. Now if $S_n$ does \emph{not} escape to the horizon, then it necessarily contains a subsequence that converges graphically to a map $S : \reals^k \toto \reals^\ell$ with non-empty domain \citep[Theorem~5.36]{Rockafellar-Wets}.

The graph of the graphical limit of a sequence of maps is closed, and maps with closed graphs are outer semicontinuous \citep[Theorem~5.7]{Rockafellar-Wets}, which is a property of the values of the map in the neighbourhood of a given point. Here, we need to extend this to neighbourhoods of a compact set.

\begin{lemma}
	\label{lem:oscK}
	If the graph of $T : \reals^k \rightrightarrows \reals^\ell$ is closed and if the compact set $K \subset \reals^k$ is such that $T(G)$ is bounded for some open set $G \supset K$, then for every $\eps > 0$ there exists $\delta > 0$ such that 
	\[
	T(K + \delta\ball) \subset T(K) + \eps\ball.
	\]
\end{lemma}

\begin{proof}[Proof of Lemma~\ref{lem:oscK}]		
	As $K \subset G$ is compact and $G \subset \reals^k$ open,  we have $K + \delta\ball \subset G$ for sufficiently small $\delta > 0$. 

Suppose the statement does not hold. Then we can find $\eps_0 > 0$, a sequence $\delta_n > 0$ tending to zero and points $x_n \in K + \delta_n\ball$ and $y_n \in T(x_n)$ such that $y_n \not\in T(K) + \eps_0\ball$. For every $n$, there exists $x_n' \in K$ such that $\norm{x_n' - x_n} \le \delta_n$. Passing to a subsequence if necessary we may assume that $x_n' \to x \in K$ and thus also $x_n \to x \in K$. Further, since $y_n \in T(G)$ for all sufficiently large $n$ and since $T(G)$ is bounded, we can, upon passing to a further subsequence, assume that $y_n \to y$. The limit point $y$ cannot belong to $T(K)$ since $y_n \not\in T(K) + \eps_0\ball$ for all $n$. We find that $(x_n, y_n) \in \gph(T)$ converges to a point $(x, y)$ not in $\gph(T)$, in contradiction to the assumption that $\gph(T)$ is closed.
\end{proof}

The following is an elaboration of a property of graphically converging maximal monotone mappings \citep[Exercise~12.40(a)]{Rockafellar-Wets}.

\begin{proposition}[Uniformity in graphical convergence]
	\label{prop:M-K} 
        Let $T_n, T : \Rd \rightrightarrows \Rd$ be maximal monotone maps and suppose that $T_{n}$ converges graphically to $T$. Let $K \subset \intr (\dom T)$ be compact. 
        \begin{enumerate}[(a)]
        \item For every $\eps> 0$, there exists $n_{\eps,K}\in\NN$ such that for all $n\geq n_{\eps,K}$ and all $A\subset K$,
        \begin{equation}        
        \begin{split}
        \label{eq:inclusion}
        T(A) &\subset T_{n}(A+\eps\ball)+\eps\ball,\\
        T_{n}(A) &\subset T(A+\eps\ball)+\eps\ball.
        \end{split}
        \end{equation}
        \item The sets $T(K)$ and $T_{n}(K)$ are compact, the latter for sufficiently large $n$, and
        \begin{equation}
        \lim_{\eps\downarrow0}\limsup_{n\to\infty}
        d_{H}\left(T_{n}(K+\eps\ball),T(K)\right)=0.
        \label{eq:Hausdorff:T}
        \end{equation}
        \item For every $x \in \intr (\dom T)$ for which $T(x)$ is a singleton and every sequence $(x_n)_n$ in $\Rd$ such that $x_n \to x$, we have $d_{H}(T_n(x_n), T(x)) \to 0$ as $n \to \infty$.
        \end{enumerate}
\end{proposition}

\begin{proof}[Proof of Proposition~\ref{prop:M-K}]
Let $x \in \intr (\dom T)$. By \citep[Theorem~2.2]{pennanen+r+t:2002}, there exists a neighbourhood $V$ of $x$ and a bounded set $B \subset \Rd$ such that $T(V) \subset B$ and $T_n(V) \subset B$ for all sufficiently large $n$. This property extends straightforwardly to a compact $K \subset \intr (\dom T)$: there exists an open set $V \supset K$ and a bounded set $B \subset \Rd$ such that $T(V) \subset B$ and $T_n(V) \subset B$ for all sufficiently large $n$. 
As the maps $T$ and $T_n$ are maximally monotone, they are closed \citep[Exercise~12.8]{Rockafellar-Wets}, hence $T(K)$ and $T_n(K)$ are compact. Furthermore, \citep[Exercise~5.34(a)]{Rockafellar-Wets} implies \eqref{eq:inclusion}.
	
	Let $\eta > 0$. By Lemma~\ref{lem:oscK}, there exists $\eps\in(0,\eta/4)$ satisfying 
	$T(K + 2\eps\ball) \subset T(K) + (\eta/2)\ball$.
	Applying (a) to the set $K + \eps\ball$, we obtain that for all sufficiently large $n$,
	\begin{align*}
	T(K)
	&\subset T_n(K + \eps\ball) + \eps\ball \\
	&\subset T(K + 2\eps\ball) + 2\eps\ball \\
	&\subset T(K) + \eta\ball,
	\end{align*}
	so the Hausdorff distance between $T_n(K + \eps\ball)$ and $T(K)$ must be bounded by $\eta$. This implies \eqref{eq:Hausdorff:T}.	
	In the special case $K = \{x\} \subset \intr(\dom T)$, this implies (c).
\end{proof}

\small

\bibliographystyle{chicago}
\bibliography{biblio}

\end{document}